%% file: main.tex
\documentclass[a4paper]{amsart}

\usepackage{amsfonts,amsmath,amssymb, amscd}
\usepackage[colorlinks=true, linkcolor=blue, citecolor=blue]{hyperref}

\include{mathscommands} 
\include{thmdefs}       

\numberwithin{equation}{section}
\title[Invariant Drinfeld twists]
{Cohomology of invariant Drinfeld twists\\
on group algebras}

\author{Pierre Guillot and Christian Kassel}

\address{
Universit\'{e} de Strasbourg \& CNRS\\
Institut de Recherche Math\'{e}matique Avanc\'{e}e\\
7~Rue Ren\'{e} Descartes\\
67084 Strasbourg, France}

\email{guillot@math.u-strasbg.fr, kassel@math.u-strasbg.fr}

\urladdr{www-irma.u-strasbg.fr/\raise-2pt\hbox{\~{}}guillot/,
www-irma.u-strasbg.fr/\raise-2pt\hbox{\~{}}kassel/}



\begin{document}

\begin{abstract}
We show how to compute a certain group~$\H^2_{\ell}(G)$ 
of equivalence classes of invariant Drinfeld twists on the algebra
of a finite group~$G$ over a field~$k$ of characteristic zero.
This group is naturally isomorphic to
the second lazy cohomology group $\H_{\ell}^2(\OO_k(G))$
of the Hopf algebra~$\OO_k(G)$ of $k$-valued functions on~$G$.
When~$k$ is algebraically closed, the answer involves
the group of outer automorphisms of~$G$ induced by conjugation in the group algebra
as well as the set of all pairs~$(A, b)$, 
where $A$ is an abelian normal subgroup of~$G$ and 
$b : \widehat{A} \times \widehat{A} \to k^{\times}$ is a $k^\times$-valued
$G$-invariant non-degenerate alternating bilinear form on the dual~$\widehat{A}$.
When the ground field~$k$ is not algebraically closed, 
we use algebraic group techniques
to reduce the computation of~$\H_{\ell}^2(G)$ 
to a computation over the algebraic closure.
As an application of our results, we compute $\H^2_{\ell}(G)$ for a number of groups.
\end{abstract}

\maketitle

\input{intro}

\input{def-lazy}

\input{groupcase}

\input{abelian}

\input{support}

\input{lazy_is_abelian}

\input{algebraic_groups}

\input{examples}

\section*{Acknowledgements}

The present work is part of the project ANR BLAN07-3$_-$183390 
``Groupes quantiques~: techniques galoisiennes et d'int\'egration" funded
by Agence Nationale de la Recherche, France.

We are grateful to Eli Aljadeff and Julien Bichon for several useful discussions.

One of the starting points of this work was a series of computer
calculations, giving in particular $\H^2_\ell(A_4) = \ZZ/2$ and
$\H^2_\ell(Q_8) = 1$ directly from the definition. Computers have also
proved useful on several other occasions. We have relied exclusively
on the open-source SAGE Computer Algebra Package, and would like to
thank its creators for keeping this wonderful software freely
available.

\input{refs}
\end{document}

%% file: mathscommands.tex
\usepackage{amsmath, amssymb, amsfonts, amscd, amsthm}

\newcommand\pf{\begin{proof}}
\newcommand\epf{\end{proof}}

\newcommand\Gr{{G}}
\newcommand\Gal{\operatorname{Gal}}
\newcommand\End{\operatorname{End}}
\newcommand\Hom{\operatorname{Hom}}
\newcommand\Aut{\operatorname{Aut}}
\newcommand\Inn{\operatorname{Inn}}
\newcommand\Int{\operatorname{Int}}

\newcommand\Ker{\operatorname{Ker}}

\newcommand\Reg{\operatorname{Reg}}
\newcommand\Lie{\operatorname{Lie}}
\renewcommand\Im{\operatorname{Im}}

\newcommand\CoInn{\operatorname{CoInn}}
\newcommand\CoInt{\operatorname{CoInt}}

\newcommand\id{{\operatorname{id}}}
\newcommand\ad{{\operatorname{ad}}}

\newcommand\op{{\operatorname{op}}}
\newcommand\cop{{\operatorname{cop}}}
\newcommand\tw{{\operatorname{tw}}}

\renewcommand\bigcup{{\operatorname{colim}\ }}

\newcommand\GL{{\operatorname{GL}}}
\newcommand\SL{{\operatorname{SL}}}
\newcommand\PSL{{\operatorname{PSL}}}

\newcommand\ZZ{\mathbb Z}

\renewcommand\H{\mathrm{H}}

\newcommand\OO{\mathcal O}

\newcommand\eps{\varepsilon}

\newcommand{\kg}{k[G]}

\newcommand{\Z}{\mathbb{Z}}
\newcommand{\C}{\mathbb{C}}

\newcommand{\F}{\mathbb{F}}
\newcommand{\B}{\mathcal{B}}

%% file: thmdefs.tex
\newtheoremstyle{pedro}{}{}{\itshape}{}{\sc}{~--}{ }{\thmname{#1}\thmnumber{ #2}\thmnote{ (#3)}}

\newtheoremstyle{pedroex}{}{}{}{}{\sc}{~--}{ }{\thmname{#1}\thmnumber{ #2}\thmnote{ (#3)}}

\theoremstyle{pedro}
\newtheorem{lem}{Lemma}[section]

\newtheorem{thm}[lem]{Theorem}

\newtheorem{prop}[lem]{Proposition}

\newtheorem{coro}[lem]{Corollary}

\newtheorem{Lem}[lem]{Lemma}
\newtheorem{Prop}[lem]{Proposition}
\newtheorem{Thm}[lem]{Theorem}
\newtheorem{Def}[lem]{Definition}

\theoremstyle{remark}

\theoremstyle{pedroex}

\newtheorem{Rem}[lem]{Remark}
\newtheorem{Expl}[lem]{Example}
\newtheorem{Not}[lem]{Notation}

%% file: intro.tex
\section*{Introduction}

The aim of this paper is to show how to compute a certain group~$\H^2_{\ell}(G)$ 
of equivalence classes of invariant Drinfeld twists on the group algebra
of a finite group~$G$. 
Drinfeld twists were introduced by Drinfeld in his work~\cite{Dr} on quasi-Hopf algebras.
They were used by Etingof and Gelaki~\cite{EG}
to classify the triangular semisimple cosemisimple Hopf algebras.
There is an abundant literature on the subject, starting
with Movshev's article~\cite{Mov} and including several papers by Davydov, Etingof, Gelaki and others;
see \cite{AEGN, Da1, Da2, Da3, EG0, EG1, EG, EG2, Ge}.

There are mainly two reasons why we restrict to \emph{invariant} Drinfeld twists:

(a) firstly, invariant twists form a group under multiplication 
whereas general twists do not;

(b) secondly, the group~$\H^2_{\ell}(G)$ we consider
identifies naturally with the second \emph{lazy cohomology} group~$\H_{\ell}^2(\OO_k(G))$
of the Hopf algebra~$\OO_k(G)$ of $k$-valued functions on~$G$:
\[
\H^2_{\ell}(G) \cong \H_{\ell}^2(\OO_k(G)) \, .
\]

The lazy cohomology group~$\H_{\ell}^2(H)$ was defined by Schauenburg~\cite{Scb}
for any Hopf algebra~$H$; 
it was systematically investigated by Bichon and Carnovale in~\cite{BC}.
The group~$\H_{\ell}^2(H)$ has been used to compute the Brauer group of a Hopf algebra
and to equip the category of projective representations of~$H$ 
with the structure of a crossed $\pi$-category in the sense of Turaev~\cite{Tu}.

When $H$ is a cocommutative Hopf algebra, then $\H_{\ell}^2(H)$
coincides with the Hopf algebra cohomology group~$H^2(H,k)$ 
introduced by Sweedler in~\cite{Sw1}.
In the special case where $H = k[G]$ is a group algebra,
$\H_{\ell}^2(H)$ is isomorphic to the cohomology group~$H^2(G,k^{\times})$
of the group~$G$ acting trivially on~$k^{\times}$.
There are very few examples of non-cocommutative Hopf algebras 
for which the lazy cohomology has been determined.
So the computation of the lazy cohomology of the function algebras~$\OO_k(G)$, 
which are in general not cocommutative, was our prime motivation. 
We also felt that such a computation might lead to a better understanding
of the corresponding ``versal torsors" constructed in~\cite{AK}.

Let us summarize our results. 
Given a group~$G$, let $\Int_k(G)$ be the group of automorphisms of~$G$ 
induced by conjugation in the group algebra~$k[G]$; 
this group contains the group~$\Inn(G)$ of inner automorphisms of~$G$.
We first show that the quotient group $\Int_k(G)/\Inn(G)$ embeds into~$\H^2_{\ell}(G)$.
As a consequence, there are finite groups~$G$ for which 
$\H^2_{\ell}(G)$ is non-abelian. This settles an open question:
there are indeed (finite-dimensional) Hopf algebras~$H$ such that
the lazy cohomology group~$\H_{\ell}^2(H)$ is \emph{not} abelian.

Next, let~$\B(G)$ be the set of pairs~$(A, b)$, 
where $A$ is an abelian normal subgroup of~$G$ and 
$b : \widehat{A} \times \widehat{A} \to k^{\times}$ is a $k^\times$-valued
$G$-invariant non-degenerate alternating bilinear form 
on the Pontryagin dual~$\widehat{A}$ of~$A$. 
When the ground field~$k$ is algebraically closed and
of characteristic zero, we construct a map
\[
\Theta : \H^2_{\ell}(G) \to \B(G)
\]
whose fibres are the left cosets of~$\Int_k(G)/\Inn(G)$; 
we show that this map is surjective if in addition the order of~$G$ is odd.
In case $\B(G)$ is trivial, we have
\[
\H^2_{\ell}(G) \cong \Int_k(G)/\Inn(G) \, .
\]
We also identify the set $\B(G)$ with a colimit (in the category of sets)
of~$H^2(\widehat{A},k^{\times})^G$
over all abelian normal subgroups~$A$ of~$G$.
As an application, if $G$ is a group of odd order with $\Int_k(G)=\Inn(G)$ and with a unique 
maximal abelian normal subgroup~$A$, then
$\H^2_{\ell}(G)$ has the following simple expression:
\[
\H^2_{\ell}(G) \cong H^2(\widehat{A},k^{\times})^G \, .
\]
Note that although these results can be expressed purely in group-theoretical terms,
their proofs rely on techniques originating in the theory of quantum groups.

We use the previous results to compute~$\H^2_{\ell}(G)$ in a number of examples.
We first give a list of groups for which $\H_{\ell}^2(G)$ is trivial:
this includes all simple groups and all symmetric groups.
We also exhibit groups for which $\H^2_{\ell}(G)$ is non-trivial:
if $G$ is the wreath product $\Z/p \wr \Z/p$ with $p$ an odd prime, then 
$\H_{\ell}^2(G)$ is an elementary abelian $p$-group of rank~$(p-1)/2$.
For the alternating group~$A_4$, we have $\H^2_{\ell}(A_4) \cong \ZZ/2$
and we give explicitly an invariant twist generating this group.
We also consider a group of order~$32$ with non-trivial $\Int_k(G) / \Inn(G)$;
we write down an invariant twist generating the subgroup of~$\H^2_{\ell}(G)$
coming from $\Int_k(G) / \Inn(G)$.

We also present a general sufficient condition for $\H^2_{\ell}(G)$ to be abelian, namely
it is so if the tensor product of any two irreducible representations
of~$G$ is a multiplicity-free direct sum of irreducible representations.

When the ground field~$k$ is not algebraically closed, 
we use Lie algebra techniques to compare the group $\H^2_{\ell}(G/k) = \H^2_{\ell}(G)$ 
to the group~$\H^2_{\ell}(G/\bar k)$ of gauge equivalence classes of invariant Drinfeld twists 
with coefficients in the algebraic closure~$\bar k$ of~$k$. When $k$ is large enough, we 
obtain an exact sequence of the form
\[
1 \longrightarrow H^1(k, Z(G)) \longrightarrow \H^2_{\ell}(G/k) 
\longrightarrow \H^2_{\ell}(G/\bar k) \longrightarrow 1 \, ,
\]
where $H^1(k, Z(G))$ is the first Galois cohomology of~$k$
with coefficients in the centre of~$G$.
This reduces the computation of~$\H_{\ell}^2(G/k)$ over an arbitrary field 
to a computation over its algebraic closure.

The paper is organized as follows.
In Section~\ref{inv-twist} we recall the definition of Drinfeld twists and
we define what we call the twist cohomology of a Hopf algebra~$H$; when $H$
is finite-dimensional we show that the twist cohomology of~$H$ is isomorphic
to the lazy cohomology of the dual Hopf algebra.

In Section~\ref{grings} we apply the content of the previous section to the case when
$H$ is the algebra of a finite group~$G$. 
We prove that $\Int_k(G)/\Inn(G)$ embeds into~$\H^2_{\ell}(G)$,
from which we deduce that $\H^2_{\ell}(G)$ is not always abelian.

In Section~\ref{abelian} we deal with the case where $G$ is a finite abelian group and 
compute $\H_{\ell}^2(G)$ in terms of alternating bilinear forms.
We also show that any abelian group with a non-degenerate alternating bilinear form
is of the form~$A \times A$.

In Section~\ref{sec_support} we present our main results involving the map~$\Theta$.
We give its main properties and identify $\B(G)$ with a colimit.

In Section~\ref{sufficient} we exploit the semi-simplicity of~$k[G]$
to deduce the above-mentioned 
condition that ensures that $\H_{\ell}^2(G)$ is abelian.

In Section~\ref{sec_rationality} we remark that the lazy cohomology of~$\H_{\ell}^2(G)$
can be computed in terms of algebraic groups. 
We determine the corresponding Lie algebras and establish the exact sequence above.

In Section~\ref{examples} we use our main theorem to 
compute $\H_{\ell}^2(G)$ for a number of specific groups.

\subsection*{Convention}

Throughout the paper, we fix a field~$k$ over which all
our constructions are defined. 
In particular, all linear maps are supposed to be $k$-linear
and unadorned tensor products mean tensor products over~$k$.

For simplicity, we assume that the ground field~$k$ is of 
\emph{characteristic zero}. 
The reader can check that all our results hold, 
except possibly those of Section~\ref{sec_rationality},
under the more general hypothesis that the characteristic of~$k$ does not
divide the orders of the groups we consider.

For any algebra~$A$ we denote the group of invertible elements of~$A$ by~$A^{\times}$.

%% file: def-lazy.tex
\section{Invariant Drinfeld twists}\label{inv-twist}

In this section we first recall what Drinfeld twists on a Hopf algebra are.
Next we use the Drinfeld twists that are invariant 
to attach to each Hopf algebra~$H$ 
two ``cohomology groups" $\H^1_{\tw}(H)$ and~$\H^2_{\tw}(H)$
in such a way that, when $H$ is finite-dimensional, these groups
are isomorphic to the lazy cohomology groups of the dual Hopf algebra.

\subsection{Twists}\label{twist}

Let $H$ be a Hopf algebra. 
A \emph{Drinfeld twist} on~$H$, or simply a \emph{twist} on~$H$, 
is an invertible element~$F$ of~$H \otimes H$ satisfying the condition
\begin{equation}\label{twist1}
(F \otimes 1) \, (\Delta \otimes \id_H)(F) = 
(1 \otimes F) \, (\id_H \otimes \Delta)(F) \, ,
\end{equation}
where $\Delta : H \to H \otimes H$ denotes systematically the coproduct of~$H$.
Twists were first defined by Drinfeld~\cite{Dr}.

The main feature of a twist~$F$ lies in the fact that it endows $H$ with a new Hopf algebra structure:
this new Hopf algebra, which we denote by~$H^F$, is equal to~$H$ as an algebra, 
has the same counit as~$H$, but has a new coproduct $\Delta^F$ given for all $a\in H$ by
\begin{equation}\label{twisted-Delta}
\Delta^F(a) = F \, \Delta(a) \, F^{-1} \, .
\end{equation}

The group $H^{\times}$ of invertible elements of~$H$ acts on the set of twists by 
\begin{equation}\label{act-on-twist}
a\cdot F= (a\otimes a) \, F \, \Delta(a ^{-1}) \, ,
\end{equation}
where $a\in H^{\times}$ and $F$ is a twist.
Using standard terminology, we say that twists belonging to the same orbit under~$H^\times$ 
are \emph{gauge equivalent}. 
Gauge equivalent twists $F, F'$ give rise to isomorphic Hopf algebras $H^F$ and~$H^{F'}$;
more precisely, if $F' = a\cdot F$, then the map $x \mapsto axa^{-1}$ induces an
isomorphism $H^F \to H^{F'}$ of Hopf algebras. 
For details, see~\cite{Dr, Mov, EG, AEGN}.

In general the product in $(H\otimes H)^{\times}$ of two twists is not a twist. 
To turn the set of twists into a group, we restrict to a special class of twists.
A twist $F$ is called \emph{invariant} if 
\begin{equation}\label{A2G}
\Delta(a) \, F = F \, \Delta(a)
\end{equation}
for all $a\in H$.
We will show below that invariant twists form a group.
Let us observe beforehand that if $F$ is an invariant twist, then 
$\Delta^F = \Delta$, where $\Delta^F$ is the twisted coproduct defined by~\eqref{twisted-Delta}.
Hence $H^F = H$ for such a twist
(nevertheless an invariant twist may change the possibly existing quasi-triangular structure of~$H$,
as we shall see in Section~\ref{subsec_braided}). 
Note also that if $a \in H$ is a group-like element, i.e., such that $\Delta(a) = a \otimes a$,
and if $F$ is an invariant twist, then $a\cdot F= F$.

In the literature twists are sometimes assumed to satisfy the additional condition
\begin{equation}\label{twist2}
(\eps \otimes \id)(F) = 1 = (\id \otimes \eps)(F) \, ,
\end{equation}
where $\eps: H\to k$ is the counit of~$H$.
We say that a twist is \emph{normalized} if Equations~\eqref{twist2} hold.

\begin{Lem}\label{lem-norm}
A twist $F$ is normalized if and only if it satisfies the condition
\begin{equation}\label{twist2bis}
(\eps \otimes \eps)(F) = 1 \, .
\end{equation}
\end{Lem}

\pf
Let $a = (\eps \otimes \id)(F)$ and $b = (\id \otimes \eps)(F)$; these are invertible elements of~$H$.
Applying $\id\otimes \eps \otimes \id$ to both sides of~\eqref{twist1}, we obtain
$(b \otimes 1) \, F = (1 \otimes a) \, F$.
Multiplying by~$F^{-1}$ and applying $\eps \otimes \id$, we obtain
$a = \eps(b) \, 1$. Similarly, applying $\id\otimes \eps$ yields
$b = \eps(a) \, 1$. 
It follows that $a, b$ are equal scalar multiples of the unit~$1$ and we have
\begin{equation*}
a = b = (\eps \otimes \eps)(F)\, .
\end{equation*}
From this the desired equivalence can be deduced easily.
\epf

\subsection{Twist cohomology}\label{twist-coho}

We now show how to construct two cohomology groups out of invariant twists. 
To this end, we consider two groups $A^1(H)$ and~$A^2(H)$ defined as follows:
$A^1(H)$ is the centre of~$H^{\times}$ and
$A^2(H)$ consists of all invertible elements of~$H \otimes H$ satisfying~\eqref{A2G}.

For $a\in A^1(H)$, set 
\begin{equation}\label{d1}
\delta^1(a)  = (a \otimes a)  \, \Delta(a^{-1}) \in (H \otimes H)^{\times} \, .
\end{equation}
For $F \in A^2(H)$, set
\begin{equation}\label{d2L}
\delta^2_L(F)  = (F \otimes 1) \,  (\Delta \otimes \id)(F)
\in (H \otimes H \otimes H)^{\times}
\end{equation}
and 
\begin{equation}\label{d2R}
\delta^2_R(F)  = (1 \otimes F) \, (\id \otimes \Delta)(F)
\in (H \otimes H \otimes H)^{\times} \, .
\end{equation}

It follows from the definitions that the equalizer
\[
Z^2(H) = \{F \in A^2(H) \, |\,  \delta^2_L(F) =  \delta^2_R(F)\}
\]
is the set of invariant twists on~$H$.

Let $A^3(H)$ be the group of invertible elements $\omegaÊ\in H \otimes H \otimes H$ satisfying
\begin{equation}\label{A3G}
\Delta^{(2)}(a) \, \omega = \omega \, \Delta^{(2)}(a)
\end{equation}
for all $a\in H$. Here $\Delta^{(2)} = (\Delta \otimes \id) \Delta  = (\id \otimes \Delta) \Delta$
stands for the iterated coproduct.

\begin{Lem}\label{lem-deltas}
(a) The map $\delta^1 : A^1(H) \to (H \otimes H)^{\times}$ is a group homomorphism
whose image~${\Im}(\delta^1)$ is a central subgroup of~$A^2(H)$.

(b) The maps $\delta^2_L, \delta^2_R : A^2(H) \to (H \otimes H \otimes H)^{\times}$
are group homomorphisms whose images are subgroups of~$A^3(H)$.

(c) The group $Z^2(H)$ is a subgroup of~$A^2(H)$ containing~${\Im}(\delta^1)$ as a central subgroup.
\end{Lem}

\pf
This is straightforward and left to the reader.
\epf

In view of Lemma~\ref{lem-deltas}, we may define the following groups.

\begin{Def}\label{def-twist-coho}
The twist cohomology groups of a Hopf algebra~$H$ are given by
\begin{equation}\label{H1G}
\H^1_{\tw}(H) = \Ker\bigl(\delta^1 : A^1(H) \to A^2(H) \bigr)
\end{equation}
and 
\begin{equation}\label{H2G}
\H^2_{\tw}(H) = Z^2(H)/{\Im}(\delta^1) \, .
\end{equation}
\end{Def}

\emph{A~priori} there is no reason why $\H^2_{\tw}(H)$ should be abelian,
and indeed in Section~\ref{Out} we shall prove the following.

\begin{prop}\label{nonabelian-Htw}
There are finite-dimensional
Hopf algebras~$H$ for which the group~$\H_{\tw}^2(H)$ is not abelian.
\end{prop}

By contrast, $\H^1_{\tw}(H)$ is abelian since $A^1(H)$ is.
The following result is an immediate consequence of the definitions.

\begin{Prop}\label{H1lG}
The group $\H^1_{\tw}(H)$ is the multiplicative group of central group-like elements of~$H$.
\end{Prop}

\begin{Rem}
Twist cohomology can also be defined using \emph{normalized} invariant twists. 
Indeed, first observe that the
normalized invariant twists form a subgroup~$\bar{Z}^2(H)$ of~$Z^2(H)$.
In view of Lemma~\ref{lem-norm}, we have the group decomposition
\[
Z^2(H) = k^{\times} (1\otimes 1) \times \bar{Z}^2(H) \, .
\]
On the other hand, we can decompose $A^1(H)$ as
\[
A^1(H) = k^{\times} 1 \times \bar{A}^1(H) \, ,
\]
where $\bar{A}^1(H)$ consists of the elements $a \in A^1(H)$ such that $\eps(a) = 1$.
It is easy to check that the homomorphism sends $\bar{A}^1(H)$ into~$\bar{Z}^2(H)$
and maps the trivial factor~$k^{\times} 1$ isomorphically onto~$k^{\times} (1\otimes 1)$.
Therefore, $\H^1_{\tw}(H)$ is equal to the kernel of $\delta^1 : \bar{A}^1(H) \to \bar{Z}^2(H)$
and 
\[
\H^2_{\tw}(H) \cong \bar{Z}^2(H)/\delta^1(\bar{A}^1(H)) \, .
\]
\end{Rem}

\subsection{Gauge equivalence for invariant twists}\label{Hopf-gauge}

Recall from Section~\ref{twist} that the group~$H^\times$ 
acts on the set of twists by~\eqref{act-on-twist}
and that twists belonging to the same orbit are called gauge equivalent. 

We may wonder which invertible elements of~$H$ preserve the set of \emph{invariant twists}.
To answer this question, 
we introduce the group~$N(H)$ of  elements $a\in H^\times$
such that
\begin{equation}\label{Hopf-auto}
\Delta(axa^{-1}) = (a \otimes a) \, \Delta(x) \, (a^{-1} \otimes a^{-1})
\end{equation}
for all $x \in H$.
For each $a\in H$ consider the algebra automorphism $\ad(a)$ of~$H$
sending~$x$ to~$axa^{-1}$. It is easy to check that an element $a\in H^\times$
satisfies Condition~\eqref{Hopf-auto} if and only if $\ad(a)$
is an automorphism of Hopf algebras.

It is easily verified that the group~$N(H)$ contains the centre~$A^1(H)$ of~$H^\times$,
the group~$\Gr(H)$ of group-like elements of~$H$, 
and their product $A^1(H) \,  \Gr(H)$ as normal subgroups.

\begin{prop}\label{prop-Hgaugequiv}
(a) For any invariant twist~$F$, the twist $a\cdot F$ is invariant
if and only if $a\in N(H)$.

(b) Let $F$ be an invariant twist.
If $F'= a \cdot F$ for some $a\in N(H)$, then $F$ and~$F'$ differ by an element 
of $\Im(\delta^1) = \delta^1(A^1(H))$ 
if and only if $a\in A^1(H) \, \Gr(H)$.

(c) The map $a \mapsto a \cdot (1\otimes 1)$ from $N(H)$ to~$Z^2(H)$
induces an injective group homomorphism
\[
N(H) / A^1(H) \,  \Gr(H) \hookrightarrow \H^2_{\tw}(H) \, .
\]
\end{prop}

\pf
(a) By definition, the twist $a\cdot F$ is invariant if and only if
\[
\Delta(x) \, (a\otimes a) \, F \, \Delta(a ^{-1}) = (a\otimes a) \, F \, \Delta(a ^{-1}) \, \Delta(x)
\]
for all $x \in H$. In view of the invariance and the invertibility of~$F$, this is equivalent to
\[
\Delta(x) (a\otimes a) \, \Delta(a ^{-1}) = (a\otimes a) \, \Delta(a ^{-1}) \, \Delta(x) \,,
\]
which in turn is equivalent to Condition~\eqref{Hopf-auto}, in which $a$ has been replaced by~$a^{-1}$;
equivalently,  $a^{-1}$, hence~$a$, belongs to~$N(H)$.

(b) There is $z \in A^1(H)$ such that $F' = \delta^1(z) \, F$ if and only if
\[
(a\otimes a) \, \Delta(a ^{-1}) = (z\otimes z) \, \Delta(z ^{-1}) \, ,
\]
which is equivalent to
$\Delta(z ^{-1} a) = z ^{-1} a\otimes z ^{-1} a$, hence to $z ^{-1} a \in \Gr(H)$.
The conclusion follows.

(c) In view of (a) and~(b) it is enough to check that the map $a \mapsto a \cdot (1\otimes 1)$
preserves products. 
We note the following equality for any $a,b \in N(H)$ and any invariant twist~$F$:
\begin{equation}\label{abF}
\bigl( a \cdot (1\otimes 1)\bigr) \, ( b \cdot F ) = (ab) \cdot F\, .
\end{equation}
This equality follows from an easy computation using the invariance of~$b \cdot F$.
To conclude, apply~\eqref{abF} to~$F = 1 \otimes 1$.
\epf

Let us consider the following subgroups of the group of Hopf algebra automorphisms of~$H$:
\[
\Int(H) = \ad\bigl(N(H)\bigr) 
\quad\text{and}\quad
\Inn(H) = \ad\bigl(\Gr(H)\bigr) \, .
\]
The group~$\Inn(H)$ is a normal subgroup of~$\Int(H)$.
Since the kernel of~$\ad$ is $A^1(H)$, 
it follows that $\ad(a)$, where $a \in N(H)$,
belongs to~$\Inn(H)$ if and only if $a$ belongs to $A^1(H) \,  \Gr(H)$.
Therefore, we have a natural isomorphism of groups
\begin{equation}\label{eq-NH-Int}
N(H) / A^1(H) \,  \Gr(H) \cong \Int(H)/\Inn(H) \, .
\end{equation}

In view of this isomorphism, 
Proposition~\ref{prop-Hgaugequiv} can be summarized as follows.

\begin{coro}\label{coro_int_innH}
(a) The group $\Int(H)/\Inn(H)$ acts freely on~$\H^2_{\tw}(H)$. 
Two elements of~$\H^2_{\tw}(H)$ belong to the same orbit under this action 
if and only if they can be represented by gauge equivalent invariant twists.

(b) Gauge equivalence and equivalence modulo~$\Im(\delta^1)$ define the same relation 
on invariant twists if and only if $\Int(H)/\Inn(H) = 1$.

(c) The group $\H^2_{\tw}(H)$ has a subgroup isomorphic to~$\Int(H)/\Inn(H)$.
\end{coro}

\subsection{Lazy cohomology}\label{lazy-cohom}

The definition of twist cohomology is motivated by lazy cohomology. 
Indeed, it is well known that the twists on a finite-dimensional Hopf algebra 
are in bijection with the two-cocycles on the dual Hopf algebra.
As we will see in Section~\ref{relating}, invariant twists correspond to
so-called \emph{lazy} two-cocycles. Using the latter,
Schauenburg~\cite{Scb} defined the lazy cohomology of a Hopf algebra.

Let us first recall the definition of lazy cohomology; we follow~\cite{BC}. 
On the dual vector space $H^* = \Hom(H,k)$ there is an associative product,
called the \emph{convolution product}. It is defined for $\lambda, \mu \in \Hom(H,k)$
by
\begin{equation}\label{convolution}
( \lambda \mu)(x) = \sum_{(x)} \, \lambda(x_1) \, \mu(x_2)
\end{equation}
for all $x\in H$. 
Here and in the sequel we make use of the Heyneman-Sweedler sigma-notation 
for the image 
\[
\Delta(x) = \sum_{(x)} \, x_1 \otimes x_2
\]
of an element~$x \in H$ under the coproduct.

Under the convolution product, $\Hom(H,k)$ becomes a monoid
whose unit is the counit~$\eps$ of~$H$.
The group of invertible elements of this monoid is denoted by~$\Reg(H)$.

An element $\lambda \in \Reg(H)$ is called \emph{lazy} if
\begin{equation}\label{lazy1}
\sum_{(x)} \, \lambda(x_1) \, x_2 = \sum_{(x)} \, \lambda(x_2) \, x_1
\end{equation}
for all $x\in H$.
The set $\Reg^1_{\ell}(H)$ of lazy elements of~$\Reg(H)$
is a central subgroup of the latter.

Since $H \otimes H$ is a Hopf algebra, we may consider the group $\Reg(H\otimes H)$
of invertible elements $\sigma \in \Hom(H\otimes H,k)$.
We define $\Reg_{\ell}^2(H)$ as the subgroup of~$\Reg(H\otimes H)$ 
consisting of all those $\sigma \in \Reg(H\otimes H)$ such that
\begin{equation}\label{lazy2}
\sum_{(x), (y)} \, \sigma(x_1 \otimes y_1)\, x_2y_2  
= \sum_{(x), (y)} \, \sigma(x_2 \otimes y_2) \, x_1y_1 \in H
\end{equation}
for all $x,y \in H$. The elements of~$\Reg_{\ell}^2(H)$ are called lazy.

A \emph{lazy left $2$-cocycle} of~$H$ is an element $\sigma \in \Reg^2_{\ell}(H)$ 
satisfying the equations
\begin{equation}\label{2cocycle}
\sum_{(x),(y)} \,\sigma(x_{1} \otimes  y_{1}) \,  \sigma(x_{2}y_{2} \otimes z) =
\sum_{(y),(z)} \, \sigma(y_{1} \otimes z_{1}) \, \sigma(x  \otimes y_{2} z_{2}) 
\end{equation}
for all $x,y,z \in H$.
We denote by~$Z^2_{\ell}(H)$ the set of lazy $2$-cocycles of~$H$.
Chen~\cite{Ch} was the first to observe that this set is a group
under the convolution product. 

For $\lambda \in \Reg(H)$, define $\partial(\lambda) \in \Reg(H \otimes H)$ for all $x,y \in H$ by
\begin{equation}\label{coboundary}
\partial(\lambda)(x\otimes y) = 
\sum_{(x),(y)} \, \lambda(x_1) \, \lambda(y_1) \, \lambda^{-1}(x_2 y_2) \, ,
\end{equation}
where $\lambda^{-1}$ is the convolution inverse of~$\lambda$.
When restricted to $\Reg^1_{\ell}(H)$, 
the map~$\lambda \mapsto \partial(\lambda)$ becomes
a group homomorphism
$$\partial : \Reg^1_{\ell}(H) \to \Reg^2_{\ell}(H) \, ,$$
whose image $B^2_{\ell}(H)$ is a central subgroup of $Z^2_{\ell}(H)$.

\begin{Def}\label{lazycoh2}
The lazy cohomology groups $\H_{\ell}^1(H)$ and $\H_{\ell}^2(H)$ are given by
\begin{equation*}
\H_{\ell}^1(H) = \Ker\bigl( \partial : \Reg^1_{\ell}(H) \to \Reg^2_{\ell}(H) \bigr)
\end{equation*}
and 
\begin{equation*}
\H_{\ell}^2(H) = Z^2_{\ell}(H)/B_{\ell}^2(H) \, .
\end{equation*}
\end{Def}

The group~$\H_{\ell}^1(H)$ is abelian since $\Reg^1_{\ell}(H)$ is.
The group $\H_{\ell}^2(H)$ classifies the
bicleft biGalois objects of~$H$ up to isomorphism, as was shown in~\cite[Th.~3.8]{BC}.

When the Hopf algebra $H$ is \emph{cocommutative}, 
then the lazy cohomology groups $\H_{\ell}^i(H)$ ($i=1,2$)
coincide with the Hopf algebra cohomology groups~$H^i(H,k)$ 
constructed by Sweedler in~\cite{Sw1}. 
In the special case where $H = k[G]$ is the algebra of a group~$G$,
it follows from~\cite[Th.~3.1]{Sw1} that for $i=1,2$, 
\begin{equation}\label{group-cohom}
\H_{\ell}^i(k[G]) \cong H^i(G, k^{\times})\, ,
\end{equation}
where $H^*(G, k^{\times})$ is the cohomology of the group~$G$ acting trivially on~$k^{\times}$.
Note that when the group $k^{\times}$ is divisible (e.g., if $k$ is algebraically closed), then
\[
\H^2_{\ell}(k[G]) \cong H^2(G, k^{\times}) \cong \Hom(H_2(G,\ZZ),k^{\times})  \, ,
\]
where $H_2(G,\ZZ)$ is the second integral homology group of~$G$.

\subsection{Relating twist cohomology to lazy cohomology}\label{relating}

We have already observed that for any Hopf algebra~$H$ 
the dual vector space~$H^*$ carries an algebra structure.
If in addition $H$ is finite-dimensional, then $H^*$ is a Hopf algebra
with coproduct given by
\[
\Delta(\lambda) (x \otimes y) = \lambda(xy) 
\]
for $\lambda \in H^*$ and $x,y \in H$.
(Here we identify $\Hom(H\otimes H,k)$ with $H^* \otimes H^*$.)
The counit of~$H^*$ is given by~$\eps(\lambda) = \lambda(1)$.

The following result relates twist cohomology to lazy cohomology.

\begin{Thm}\label{HlH}
For each finite-dimensional Hopf algebra~$H$ we have
$$\H_{\tw}^1(H) \cong \H_{\ell}^1(H^*) \quad \text{and} \quad
\H_{\tw}^2(H) \cong \H_{\ell}^2(H^*) \, .$$
\end{Thm}

As a consequence of the theorem and of Proposition~\ref{nonabelian-Htw},
we obtain the following.

\begin{coro}\label{nonabelian}
There are finite-dimensional
Hopf algebras~$H$ for which the group~$\H_{\ell}^2(H)$ is not abelian.
\end{coro}

\pf[Proof of Theorem~\ref{HlH}]
Identify an element $a\in H$ with $\lambda_a : H^* \to k$ given by
$\lambda_a(\alpha) = \alpha(a)$ for all $\alpha \in H^*$.
Similarly identify an element $F \in H \otimes H$ with $\sigma_F : H^* \otimes H^* \to k$
given by
$\sigma_F(\alpha\otimes \beta) = (\alpha\otimes \beta)(F)$ for all $\alpha, \beta \in H^*$.
It is easy to check that $a\in H$ is invertible if and only if $\lambda_a \in \Reg(H^*)$
and that $F\in H \otimes H$ is invertible if and only if $\sigma_F \in \Reg(H^* \otimes H^*)$.
It is also well known that $F \in H\otimes H$ is a twist if and only if
$\sigma_F$ satisfies~\eqref{2cocycle}.
We conclude the proof with the help of the following lemma
whose proof is straightforward.
\epf

\begin{Lem}
(a) Under the previous identifications,
the map 
\[
\partial : \Reg(H^*) \to \Reg(H^* \otimes H^*)
\]
of~\eqref{coboundary}
coincides with the map $\delta^1: H^{\times} \to (H \otimes H)^{\times}$.

(b) The linear map $\lambda_a \in \Reg(H^*)$ is lazy if and only if $a \in H^{\times}$ is central.

(c) The linear map $\sigma_F \in \Reg(H^*\otimes H^*)$ is lazy if and only if  $F$ is invariant.
\end{Lem}

\begin{Rem}
Bichon and Carnovale \cite[Sect.~1]{BC} associated to each Hopf algebra~$H$ 
certain groups $\CoInn(H)$ and $\CoInt(H)$ of Hopf algebra automorphisms of~$H$.
If $H$ is finite-dimensional and $H^*$ is the dual Hopf algebra,
then $\CoInt(H) \cong \Int(H^*)$ and $\CoInn(H) \cong \Inn(H^*)$
(the groups of Hopf algebra automorphisms of~$H$ and~$H^*$ are
naturally isomorphic). 
There is a version of Corollary~\ref{coro_int_innH} in this context; 
see~\cite[Lemmas~1.11--1.13]{BC}.
In~\cite[Ex.~7.5]{BC} a $32$-dimensional (non-commutative, non-cocommutative)
Hopf algebra with non-trivial~$\CoInt(H)/\CoInn(H)$ was constructed.
\end{Rem}

%% file: groupcase.tex
\section{Hopf algebras of groups}\label{grings}

In this section we apply the content of Section~\ref{inv-twist}
to the main case of interest to us, namely when $H$ is the algebra of a group.

\subsection{The Hopf algebra~$k[G]$}\label{AG}

Let $G$ be a group and $k[G]$ the corresponding group algebra over the ground field~$k$.
It is well known that~$H = k[G]$ carries the structure of a Hopf algebra
with coproduct~$\Delta$, counit~$\eps$, and antipode~$S$ given by
\begin{equation*}\label{HopfkG}
\Delta(g) = g\otimes g\, , \quad
\eps(g) = 1\, , \quad S(g) = g^{-1}
\end{equation*}
for all $g\in G$.
The only group-like elements of~$H$ are the elements of~$G$.

In this case $A^1(H)$ is equal to the centre of~$k[G]^{\times}$ and $A^2(H)$ is the group
of invertible elements of the algebra $(k[G] \otimes k[G])^G$ of elements of~$k[G] \otimes k[G]$
fixed under the diagonal conjugation action, i.e., 
$A^2(H)$ consists of the invertible $G$-invariant elements of~$k[G] \otimes k[G]$.
Similarly, $A^3(H)$ is the group of invertible $G$-invariant elements 
of~$k[G] \otimes k[G] \otimes k[G]$.

\begin{Not}
We shall henceforth denote $\H_{\tw}^i(k[G])$ 
($i=1,2$) of Definition~\ref{def-twist-coho}
by~$\H_{\ell}^i(G)$, or by~$\H_{\ell}^i(G/k)$ if we need to stress the ground field~$k$,
as will be the case in Section~\ref{sec_rationality}.
\end{Not}

In view of Proposition~\ref{H1lG},
$\H_{\ell}^1(G)$ coincides with the centre of~$G$;
in particular, it is independent of the ground field~$k$.
So the main question is to determine the group~$\H^2_{\ell}(G)$.

\begin{Expl}
Let us consider a situation where $\H^2_{\ell}(G)$ can be completely determined
(and in fact is trivial).
We say that a group algebra~$k[G]$ has \emph{trivial units} if all invertible elements
of~$k[G]$ are of the form~$\lambda g$, where $\lambda \in k^{\times}$ and $g\in G$.
In other words, there is a group isomorphism $k[G]^{\times} \cong k^{\times} \times G$.
Using the natural projections $G\times G \to G$, one easily checks that
$k[G] \otimes k[G] = k[G \times G]$ has trivial units if $k[G]$ has trivial units.
The class of groups whose group algebras have trivial units includes all
left-orderable groups such as the free abelian groups, the free groups, the (pure) braid groups.

We now claim that if $k[G]$ has trivial units, then 
\[
\H^2_{\ell}(G) = 1 \, .
\]
Indeed, 
since $k[G] \otimes k[G] = k[G \times G]$ has trivial units, any element~$F \in A^2(k[G])$
is of the form $F = \lambda \, g \otimes h$, where $\lambda \in k^{\times}$ and $g,h\in G$.
One checks that
\begin{equation*}
\delta^2_L(F) = \lambda^2 \, g^2 \otimes hg \otimes h 
\quad\text{and} \quad
\delta^2_R(F) = \lambda^2 \, g \otimes gh \otimes h^2 \, .
\end{equation*}
It follows that any element of~$Z^2(k[G])$ is of the form $\lambda \, 1 \otimes 1$,
which belongs to the image of~$\delta^1$ since it is equal to $\delta^1(\lambda \, 1)$.

Observe that, if $k[G]$ has trivial units, then $G$ is necessarily torsion free, 
hence infinite. 
\end{Expl}

If $G$ is a \emph{finite} group, then
the Hopf algebra~$k[G]$ is finite-dimensional and we may consider its dual Hopf algebra.
The latter is the algebra~$\OO_k(G)$ of $k$-valued functions on~$G$. 
For each $g\in G$ let $e_g$ be the characteristic function of the singleton~$\{g\}$. 
The set $(e_g)_{g\in G}$ is a basis of~$\OO_k(G)$, 
dual to the basis $(g)_{g\in G}$ of~$k[G]$; 
it is also an orthogonal family of primitive idempotents.
The coproduct~$\Delta$, counit~$\eps$, and antipode~$S$ of~$\OO_k(G)$ are given by
\begin{equation}\label{HopfkGdual}
\Delta(e_g) = \sum_{h\in G} \, e_h\otimes e_{h^{-1}g}\, , \quad
\eps(e_g) = \delta_{g,1} \, , \quad S(e_g) = e_{g^{-1}}
\end{equation}
for all $g\in G$.

By Theorem~\ref{HlH} we have
\begin{equation}\label{HlG}
\H_{\ell}^i(G)  = \H_{\tw}^i(k[G]) \cong \H_{\ell}^i(\OO_k(G)) \qquad (i=1,2) \, .
\end{equation}

The constructions of Section~\ref{twist-coho} applied to the group algebra~$k[G]$
thus give a direct definition of the lazy cohomology of the function Hopf algebra~$\OO_k(G)$;
we shall use this direct definition to compute~$\H_{\ell}^2(\OO_k(G))$
in Section~\ref{sec_support}.

\subsection{Automorphism groups}\label{Out}

We now apply the content of Section~\ref{Hopf-gauge} to the case~$H = k[G]$.
If $a \in N(k[G])$, then for all $g\in G$,
\begin{eqnarray*}
\Delta(aga^{-1}) & = & (a \otimes a) \, \Delta(g) \, (a^{-1} \otimes a^{-1}) \\
& = & (a \otimes a) \, (g \otimes g) \, (a^{-1} \otimes a^{-1}) \\
& = & (aga^{-1} \otimes aga^{-1}) \, .
\end{eqnarray*}
Hence, $aga^{-1}$ belongs to~$G$ for all~$g$.
In other words, $a$ belongs to the normalizer~$N_k(G)$ of~$G$ in~$\kg^\times$. 
The converse holds and we have
\begin{equation}\label{NG=NG}
N(k[G]) = N_k(G) \, .
\end{equation}

Since $\Gr(k[G]) = G$, the group $\Inn(k[G])$ coincides with the group~$\Inn(G)$
of inner automorphisms of~$G$.
We set $\Int_k(G) = \Int(k[G])$; this is the group of automorphisms of~$G$
induced by conjugation by invertible elements of~$k[G]$. 
As a special case of~\eqref{eq-NH-Int}, we have the isomorphism
\begin{equation}\label{eq-NG-Int}
N_k(G) / A^1(k[G]) \, G \cong \Int_k(G)/\Inn(G) \, .
\end{equation}

In the present context Corollary~\ref{coro_int_innH} takes the following form.

\begin{coro}\label{coro_int_inn}
(a) The group $\Int_k(G)/\Inn(G)$ acts freely on~$\H^2_{\ell}(G)$. 
Two elements of~$\H^2_{\ell}(G)$ belong to the same orbit under this action 
if and only if they can be represented by gauge equivalent invariant twists.

(b) Gauge equivalence and equivalence modulo~$\Im(\delta^1)$ define the same relation 
on invariant twists if and only if $\Int_k(G)/\Inn(G) = 1$.

(c) The group $\H^2_{\ell}(G)$ has a subgroup isomorphic to~$\Int_k(G)/\Inn(G)$.
\end{coro}

In Section~\ref{BG} we shall check that $\Int_k(G)/\Inn(G) = 1$
for a number of classical groups.
Actually, finding finite groups with non-trivial $\Int_k(G)/\Inn(G)$ is not so straightforward. 
The easiest way to obtain such a group is via the group~$\Aut_{\rm c}(G)$
of \emph{class-preserving automorphisms} of~$G$, 
i.e., the ones that preserve each conjugacy class.
Let us explain how $\Int_k(G)$ and~$\Aut_{\rm c}(G)$ are related
when $G$ is a finite group.
First, observe that if $k \subset \bar k$ is an extension of fields, then
$\Int_k(G)$ is a subgroup of~$\Int_{\bar k}(G)$.
Now, by~\cite[Prop.~2.5]{JM}, if $\bar k$ is big enough (for instance,
if it is algebraically closed) and its characteristic is prime to the
order of~$G$, then
\begin{equation}\label{IntAut}
\Int_{\bar k}(G) = \Aut_{\rm c}(G)\, .
\end{equation}
The inclusion $\Int_k(G) \subset \Aut_{\rm c}(G)$ follows immediately
for any field~$k$ whose characteristic is prime to the order of~$G$.
This inclusion can also be deduced from the following simple argument:
by definition, for any $\varphi \in \Int_k(G)$ and any~$g\in G$,
the image $\rho(\varphi(g))$ under a representation~$\rho$ of the group 
is conjugate to~$\rho(g)$; hence, $\chi(\varphi(g)) = \chi(g)$
for any character~$\chi$ of~$G$ and any $g\in G$; it follows that $\varphi$
preserves each conjugacy class.

Burnside~\cite{Bur} was the first to exhibit examples of finite groups
with non-trivial $\Aut_{\rm c}(G)/ \Inn(G)$, hence such that 
$\Int_k(G) /\Inn(G) \neq 1$ for large enough~$k$.
His smallest example was of order~$729$;
Wall~\cite{Wal} much later found one of order~$32$
(we shall discuss this example in Section~\ref{Wall}). 
Note that the triviality of~$\Int_k(G)/\Inn(G)$ 
is related to the isomorphism problem for group algebras (see~\cite{Maz, RZ}).

We are now ready to prove Proposition~\ref{nonabelian-Htw}.

\begin{proof}[Proof of Proposition~\ref{nonabelian-Htw}]
It is enough to find a finite group such that $\H_{\ell}^2(G)$ is non-abelian.
Now, C.-H.\ Sah~\cite{Sah} constructed for each prime~$p$ 
and each integer $m\geq 3$ a $p$-group $G$ of order~$p^{5m}$ 
for which $\Aut_{\rm c}(G)/\Inn(G)$ is non-abelian.
It follows from~\eqref{IntAut} and from Corollary~\ref{coro_int_inn}\,(c) that
for such a group~$G$ (and an appropriate ground field~$k$) 
the group~$\H^2_{\ell}(G)$ is non-abelian. 
\end{proof}

Note that the smallest example exhibited by Sah is of order~$2^{15}$, that is
roughly a thousand times bigger than the smallest group 
for which $\Aut_{\rm c}(G)/\Inn(G) \neq 1$.

%% file: abelian.tex
\section{The abelian case}\label{abelian}

In this section we further assume that the ground field~$k$ is \emph{algebraically closed}.
Under this assumption we describe $\H^2_{\ell}(A)$ for any finite abelian group~$A$. 
It is worth indulging in the details for we shall see later that, in the case of a general finite group, 
its abelian subgroups control the invariant twists.

All twists for~$A$ are obviously invariant, so what we describe is simply the 
group of twists modulo gauge equivalence.
Observe also that $\Int_k(A) = \Inn(A) = \{\id\}$ in this case; 
it follows from Corollary~\ref{coro_int_inn}
that gauge equivalence coincides with equivalence modulo~$\Im(\delta^1)$.

\subsection{General picture}\label{subsec-fourier}
Let $A$ be a finite \emph{abelian} group 
and $\widehat{A} = \Hom(A,k^{\times})$ be its group of characters.
Recall that the \emph{discrete Fourier transform} is the map
\[
k[A] \to \OO_k(\widehat{A}) \, ; \, g \mapsto \widehat{g}
\]
defined for all $g\in A$ and $\chi \in \widehat{A}$ by
$\widehat{g}(\chi) = \chi(g^{-1})$.
It is easy to check that it is a Hopf algebra isomorphism.
We now compute $\H^2_{\ell}(A)$.

\begin{Prop}\label{prop_abelian_case}   
We have
$\H^2_{\ell}(A) \cong H^2(\widehat{A}, k^{\times})$.
\end{Prop}

\pf
We have the sequence of isomorphisms
\[\H^2_{\ell}(A) \cong \H^2_{\ell}(\OO_k(A)) \cong
\H^2_{\ell}(k[\widehat{A}]) \cong H^2(\widehat{A}, k^{\times}) \, .
\]
The first isomorphism follows from~\eqref{HlG}, 
the second one from the discrete Fourier transform,
and the last one from~\eqref{group-cohom}.
\epf

\subsection{Explicit formulas}\label{subsec_explicit_abelian}
The easily-remembered isomorphism of Proposition~\ref{prop_abelian_case} can be made explicit. 
For the proof we used an isomorphism between $A$ and~$\widehat{A}$. 
Following~\cite[Prop.~3]{Mov}, 
we exhibit a direct isomorphism between $\H^2_{\ell}(A)$ and $H^2(\widehat{A}, k^\times)$
as follows.

For $\rho\in \widehat{A}$, let $e_\rho$ be the corresponding idempotent in~$k[A]$. 
To a $k^{\times}$-valued two-cocycle~$c$ on~$\widehat{A}$ we associate the element 
\[
F= \sum_{\rho, \sigma} \, c(\rho, \sigma) \, e_\rho \otimes e_\sigma \in k[A] \otimes k[A] \, .
\]
Then $F$ is a twist for $k[A]$.
Conversely, given a twist $F$, define $c$ by 
\[
c(\rho, \sigma)= (\rho \otimes \sigma)(F) \, .
\] 
Then $c$ is a $k^{\times}$-valued two-cocycle on~$\widehat{A}$.

It is easily checked that gauge equivalence of twists corresponds 
to cohomological equivalence of two-cocycles. 
Moreover, the group of two-cocycles~$c$ 
modulo coboundaries is equal to~$H^2(\widehat{A}, k^\times)$.
Thus we have the desired explicit isomorphism.

Next, we note that since $k^{\times}$ is divisible, 
the universal coefficient theorem and a well-known computation of 
the second cohomology of an abelian group imply that
\[
H^2(\widehat{A}, k^{\times}) 
\cong \Hom(H_2(\widehat{A},\ZZ),k^{\times}) 
\cong \Hom(\Lambda_{\ZZ}^2 \widehat{A},k^{\times}) \, .
\]

As a result, with our assumption on $k$, the second cohomology group 
turns out to be the group of $k^{\times}$-valued alternating bilinear forms on~$\widehat{A}$. 
This can be made explicit as follows: 
the alternating form corresponding to a two-cocycle~$c$ is 
\[
b(\rho, \sigma)= \frac{c(\sigma, \rho)}{c(\rho, \sigma)} \, ,
\]
where $\rho, \sigma \in \widehat{A}$.
Thus $b$ measures the symmetry default of~$c$. 

Just as $c$ defines a twist $F \in k[A] \otimes k[A]$, the bilinear form $b$ allows us to introduce 
\begin{equation}\label{Rb}
R(A,b)= \sum_{\sigma, \tau \in \widehat{A}} \, b(\sigma, \tau) \, e_\sigma \otimes e_\tau \,  .
\end{equation}
The element $R(A,b) \in k[A] \otimes k[A]$ is invertible and we have 
$R(A,b)= F_{21}  F^{-1}$. 
The two-tensor $R(A,b)$ determines~$F$ up to gauge equivalence as follows.

\begin{prop}
If $F, F'$ are twists for $k[A]$ such that $F'_{21}  F'{}^{-1} = F_{21}  F^{-1}$, 
then $F$ and $F'$ are gauge equivalent.
\end{prop}

In Section~\ref{subsec_braidings} we shall give an analogous result for an arbitrary group.

\pf
The result can be deduced from the proof of~\cite[Lemma~5.7]{Ge}; here is a direct argument.
Consider the twist $f = F^{-1} F'$. It is symmetric in the sense that $f_{21} = f$.
The cocycle corresponding to~$f$ has an associated alternating form~$b_0$
such that $R(A, b_0)= f_{21}  f^{-1}$. Since $f$ is symmetric, we have $R(A, b_0) = 1 \otimes 1$, 
which implies that $b_0(\sigma, \tau) = 1$ for all $\sigma, \tau \in \widehat{A}$. 
Therefore, under the isomorphism of Proposition~\ref{prop_abelian_case}, the twist~$f$
represents the trivial element of~$\H^2_{\ell}(A)$. 
Consequently, $F$ and $F'$ represent the same element of~$\H^2_{\ell}(A)$,
which by Proposition~\ref{prop-Hgaugequiv} and the equality $\Int_k(A) = \Inn(A)$
yields the conclusion.
\epf

We now end these computational remarks with a comment on the inverse bijection, 
taking a $k^{\times}$-valued alternating bilinear form~$b$ on~$\widehat{A}$ 
to the class of a two-cocycle~$c$ modulo coboundaries. 
One observes that $b$ can itself be viewed as a two-cocycle on~$\widehat{A}$ whose
corresponding bilinear form is~$b^2$. So, in order to find~$c$, 
we look for a ``square root'' of~$b$. 
This is easily achieved when $A$, hence~$\widehat{A}$, has odd order: 
in this case $\rho\mapsto \rho/2$ is an automorphism of~$\widehat{A}$ and we may set 
\[
	c(\rho, \sigma) = b(\rho/2, \sigma) = b(\rho, \sigma/2) = {b(\rho, \sigma)}^{1/2} \, .
\]
for $\rho, \sigma \in \widehat{A}$.
The square root in the previous formula is uniquely defined since $b$ takes its values 
in a group of roots of unity of odd order. 
It is trivial to check that the bilinear form corresponding to this
cocycle~$c$ is indeed~$b$.

\subsection{Non-degenerate bilinear forms \& cocycles}

Let $b$ be a $k^{\times}$-valued alternating bilinear form on $\widehat{A}$. 
By~\eqref{Rb} it defines an element $R(A,b) \in k[A] \otimes k[A]$. 
We shall say that $b$ is {\em minimal} if there are no proper subgroups $B \subset A$ 
such that $R(A,b)\in k[B] \otimes k[B]$.

\begin{prop}\label{prop_minimal_is_degenerate}
The alternating bilinear form $b$ is minimal if and only if it is non-degenerate.
\end{prop} 

One way to see this is by considering the map $f : \OO_k(A)^{\cop} \to
k[A]$ defined by $R(A, b)$, that is by $f(x) = (x \otimes \id)(R(A,
b))$. It easy to see that $R(A, b)$ is minimal if and only if $f$ is
surjective, and that $b$ is non-degenerate if and only if $f$ is
injective. For reasons of dimension, these conditions are equivalent.

Alternatively, Proposition \ref{prop_minimal_is_degenerate} can
deduced from the following lemma, which we will need later anyway. It
asserts that the constructions of the previous paragraph are in a
sense natural in~$A$.

\begin{lem}\label{lem-restriction}
Let $B$ be any subgroup of $A$, and let $p: \widehat{A} \to \widehat{B}$ be the restriction map. 
Given a alternating bilinear form~$b'$ on~$\widehat{B}$, 
define $b(\sigma, \tau)= b'(p(\sigma), p(\tau))$ on~$\widehat{A}$. 
Then $R(B, b') = R(A,b)$ in~$k[A] \otimes k[A]$.

Conversely, if $b$ is any alternating bilinear form on $\widehat{A}$ 
such that $R(A,b)$ belongs to $k[B] \otimes k[B]$, 
then $b(\sigma, \tau) = b'(p(\sigma), p(\tau))$ for some form $b'$ on $\widehat{B}$.
\end{lem} 

\begin{proof}
We have 
\[
R(A,b) = \sum_{\sigma, \tau \in \widehat{A}} \, b(\sigma, \tau) \, e_\sigma \otimes e_\tau
\quad\text{and}\quad
R(B, b') = \sum_{\kappa, \lambda \in \widehat{B}} \, b'(\kappa, \lambda) \, e_\kappa \otimes e_\lambda \, .
\]
The equality $R(B, b') = R(A,b)$ follows from 
\[
e_\kappa = \sum_{\sigma \in p ^{-1} (\kappa)}  \, e_\sigma \, .
\]
This last equality is obtained by an elementary computation with characters. 
It also makes the converse clear, for it is now apparent that $R(A,b)$ lies in $k[B] \otimes k[B]$ 
if and only if $b$ is constant on the fibres of~$p$.
\end{proof}

It is customary to talk about \emph{non-degenerate two-cocycles} 
when the associated bilinear form~$b$ is non-degenerate. 
We shall also speak of \emph{minimal two-cocycles} in the same sense. 
It turns out that the existence of a non-degenerate two-cocycle 
imposes severe restrictions on~$A$, as is shown in the following proposition, 
which is well-known (see \cite[Thm.~4.8]{Kar} or \cite[Lemma~5.2]{Da3}).

\begin{prop}\label{nondeg_square}
Any finite abelian group $A$ with a 
non-degenerate alternating bilinear form $b : A \times A \to k^{\times}$
is of symmetric type, i.e., of the form $A\cong A' \times A'$.
\end{prop}

%% file: support.tex
\section{The main theorem}\label{sec_support}

In this section we present our main results.
We again assume that the ground field~$k$ is algebraically closed.

The material in Sections~\ref{subsec_braided}, \ref{subsec_braidings}, \ref{Ronfunctionalg}
is known to the experts (see for example the unpublished manuscript~\cite{Da1}).
We include it for convenience.

\subsection{Quasi-triangular Hopf algebras}\label{subsec_braided}

Let $H$ be a Hopf algebra with coproduct~$\Delta$. Denote the 
\emph{opposite coproduct} of~$H$ by~$\Delta^{\op}$; 
by definition $\Delta^{\op}(a) = \Delta(a)_{21}$ for all $a \in H$.

Let $H$ be a Hopf algebra. A \emph{universal $R$-matrix} for~$H$
is an invertible element $R\in H\otimes H$ such that
\begin{equation}\label{R1}
\Delta^{\op}(a) =  R \, \Delta(a) \, R ^{-1}
\end{equation}
for all~$a \in H$, and
\begin{equation}\label{R2}
(\Delta \otimes \id_H)(R) = R_{13} \, R_{23}
\quad\text{and}\quad
(\id_H \otimes \Delta)(R) = R_{13} \, R_{12} \, .
\end{equation}
A \emph{quasi-triangular Hopf algebra} is a Hopf algebra equipped
with a universal $R$-matrix.
For more, see~\cite[Chap.~VIII]{Kas}.

Given a universal $R$-matrix~$R$ for~$H$, it is possible to construct another one
with the help of a twist on~$H$. More precisely, 
for any twist $F$ on~$H$, set  
\begin{equation*}
R_F= F_{21} \, R \, F^{-1} \, .
\end{equation*}
Then $R_F$ is a universal $R$-matrix for the Hopf algebra~$H^F$ 
obtained from~$H$ by twisting the coproduct following~\eqref{twisted-Delta}.
If the twist $F$ is \emph{invariant}, then as observed in Section~\ref{twist},
we have $\Delta^F = \Delta$ and therefore $H^F = H$. 
In this case $R_F$ is another universal $R$-matrix for~$H$.

If $F, F'$ are gauge equivalent invariant twists on~$H$, then the corresponding 
universal $R$-matrices $R_F, R_{F'}$ are conjugate. 
Indeed, if $F' = a \cdot F$ for some invertible element $a\in H$, then
\begin{equation}\label{gaugeR}
R_{F'} = (a\otimes a) \, R_F \, (a\otimes a)^{-1} \, .
\end{equation}
It follows that, if $F$ and $F'$ represent the same element of~$\H^2_{\tw}(H)$,
i.e., if $F' = a \cdot F$ for some invertible central element~$a$,
then $R_{F'} = R_F$.

In the sequel we shall also make use of Radford's work on \emph{minimal quasi-triangular Hopf algebras}:
in~\cite{Rad} it was proved that any quasi-triangular Hopf algebra~$H$ 
with universal $R$-matrix $R = \sum_i \, s_i \otimes t_i$ contains a unique smallest 
(finite-dimensional)
quasi-triangular Hopf subalgebra $H\langle R \rangle$ with the same universal $R$-matrix~$R$
(it is the smallest Hopf subalgebra of~$H$ containing the elements~$s_i$ and~$t_i$.)
He also proved that if $H$ is either commutative or cocommutative, then
$H\langle R \rangle$ is \emph{bicommutative}, i.e., both commutative and cocommutative;
see \cite[Prop.~13]{Rad}.

\subsection{Quasi-triangular structures on group algebras}\label{subsec_braidings}

Consider the case where $H= \kg$ is the algebra of a finite group~$G$.
Since $k[G]$ is cocommutative, we can equip it with the
trivial universal $R$-matrix $1\otimes 1$.
It follows from the discussion in Section~\ref{subsec_braided}
that if $F$ is a $G$-invariant twist on~$\kg$, then
\begin{equation}\label{twistedR}
R_F = F_{21} \, (1\otimes 1) \, F^{-1} = F_{21} \, F^{-1} 
\end{equation}
is a universal $R$-matrix for~$\kg$.  We have observed above that if
$F$ and $F'$ represent the same element of~$\H^2_{\tw}(H) =
\H^2_{\ell}(G)$, then $R_{F'} = R_F$.  Conversely, the universal
$R$-matrix $R_F = F_{21} \, F^{-1}$ determines~$F^{-1}$ up to gauge
equivalence.

\begin{prop}\label{RFdeterminesF}
When $F$ and $F'$ are invariant twists, then $R_{F'} = R_F$ if and
only if $F^{-1}$ and $F'{}^{-1}$ are gauge equivalent.
\end{prop}

\pf We start with the ``only if'' part. Since $F, F'$ are invariant
twists, so is $f = F^{-1} F'$. The equality $R_{F'} = R_F$ implies
that $f$ is a symmetric twist, i.e., $f_{21} = f$.  It follows
from~\cite[Cor.~3]{EG} that $f$ is gauge equivalent to~$1\otimes 1$.
In other words, there is an invertible element $a\in H$ such that $f =
(a\otimes a) \, \Delta(a^{-1})$.  Therefore,
\[
F^{-1}  = f F'{}^{-1} 
= (a\otimes a) \, \Delta(a^{-1}) \, F'{}^{-1}
= (a\otimes a) \, F'{}^{-1} \, \Delta(a^{-1}) \, .
\]
It follows that $F^{-1}$ and $F'{}^{-1}$ are gauge equivalent.

The converse is trivial: if $F^{-1} = a \cdot F'{}^{-1}$ as above, 
then 
\[
R_F= \Delta(a) \, R_{F'} \, \Delta(a^{-1}) = R_{F'}
\]
by invariance of~$R_{F'}$.
\epf

For any quasi-triangular Hopf algebra~$H$ with universal $R$-matrix $R= \sum\, s_i \otimes t_i$, 
the \emph{Drinfeld element}~$u_R$ is defined to be 
\begin{equation}\label{def-u}
u_R = \sum\, S(t_i)s_i \, ,
\end{equation}
where $S$ is the antipode of~$H$ (see~\cite[Sect.~2]{Dr1}, \cite[Sect.~VIII.4]{Kas}). 
We need the following lemma, which has been observed in the proof
of~\cite[Prop.~3.4]{EG2}. We give a proof for the sake of completeness.

\begin{lem}\label{lem-u}
Let $G$ be a finite group and~$F$ an invariant twist on~$k[G]$.
If $R_F = F_{21} \, F^{-1}$, then
$u_{R_F} = 1$.
\end{lem}

\pf 
Any Drinfeld element $u_R$ associated to a universal $R$-matrix~$R$ satisfies the relations
\begin{equation}\label{eqn-u}
S^2(a) = u_R \, a \, u_R^{-1} \quad\text{and}\quad
\Delta(u_R) = (R_{21}R)^{-1} \, (u_R \otimes u_R)
\end{equation}
for all $a \in H$; see~\cite[Prop.~2.1 and~3.2]{Dr1} or \cite[Prop.~VIII.4.1 and~VIII.4.5]{Kas}.
The antipode of any cocommutative Hopf algebra being involutive, it follows
from the first equation in~\eqref{eqn-u} that $u_{R_F}$ is a central element of~$k[G]$.
Moreover, for $R = R_F$, we have
\[
R_{21}R = F \, F^{-1}_{21} \, F_{21} \, F^{-1} = 1 \otimes 1 \, .
\]
It then follows from the second equation in~\eqref{eqn-u}
that $\Delta(u_{R_F}) = u_{R_F} \otimes u_{R_F}$.
In other words, $u_{R_F}$ is a central group-like element of~$k[G]$,
hence a central element of the group. 
To prove that $u_{R_F} = 1$, we apply \cite[Th.~3.4]{AEGN} to the case
where $H = k[G]$ and $J = F^{-1}$. Since $k[G]$ is unimodular (and cocommutative),
this theorem states that if $\lambda \in k[G]^*$ is a non-zero integral,
then so is 
$u_{R_F} \rightharpoonup \lambda = \sum_{\lambda} \, \lambda_1 \, \lambda_2(u_{R_F})$.
Now by definition of an integral in~$k[G]^*$, we have
$g \lambda(g) = \lambda(g) 1$ for all $g \in G$.
It follows that $\lambda = \gamma \, e_1$, 
where $e_1$ is the characteristic function of the unit in~$G$ and $\gamma$ is some scalar. 
We take $e_1$ as a non-zero integral in~$k[G]^*$. Now 
\[
(u_{R_F} \rightharpoonup e_1)(g) = \sum_{(e_1)} \, (e_1)_1(g) \, (e_1)_2(u_{R_F})
= e_1 (gu_{R_F})
\]
for all $g\in G$, which means that $u_{R_F} \rightharpoonup e_1$ 
is the characteristic function of the singleton~$\{u_{R_F}^{-1}\}$.
It is then clear that the latter can be an integral in~$k[G]^*$ only if~$u_{R_F}=1$.
\epf

For any universal $R$-matrix for~$k[G]$, we may consider the
minimal quasi-triangular Hopf subalgebra~$k[G]\langle R \rangle$
introduced in Section~\ref{subsec_braided}.
Since $k[G]$ is commutative, $k[G]\langle R \rangle$
is necessarily bicommutative.
Hence, $k[G]\langle R \rangle$ must be the group algebra of an abelian subgroup~$A$ of~$G$:
\[
k[G]\langle R \rangle = k[A] \, .
\]
Now, $k[A]$ is isomorphic to the algebra~$\OO_k(\widehat{A})$ 
of functions on~$\widehat{A} = \Hom(A,k^{\times})$, 
the braidings of which have a particularly simple form. 
We proceed to explain this next.

\subsection{Quasi-triangular structures on function algebras}\label{Ronfunctionalg}

We shall consider the Hopf algebra~$H$ of functions 
(resp.\ smooth functions, resp.\ regular functions,\dots) 
on a group~$G$ (resp.\ Lie group, resp.\ algebraic group,\dots). 
In this case a universal $R$-matrix is none other than a function~$b$ on $G \times G$
and the equations~\eqref{R1} and~\eqref{R2} take a particularly simple form. 
Equation~\eqref{R1} does not involve~$b$ since $H$ is commutative; 
it actually forces $H$ to be cocommutative so that the group~$G$ must be abelian. 
Equations~\eqref{R2} become the functional equations
\[
b(xy, z)= b(x,z) \, b(y,z)
\quad\text{and}\quad
b(x, yz) = b(x,y) \, b(x,z)
\]
for all $x,y,z \in G$.
Thus $b$, which is assumed to be invertible, is just a $k^\times$-valued bilinear form 
on the abelian group~$G$. In view of the material in Section~\ref{subsec_explicit_abelian}, 
we may ask when this form is alternating. 
We give a condition involving the Drinfeld element $u_R$ defined by~\eqref{def-u}.
It follows from this definition that $u_R$ becomes the function
\[
u_R(x) = b(x ^{-1}, x) = b(x,x)^{-1}
\]
for all $x\in G$. Therefore, $u_R = 1$ implies that $b$ is alternating, and conversely.

In the case where $A$ is an abelian group and $G = \widehat{A}$, we can
use the Fourier transform, as introduced in Section~\ref{subsec-fourier} 
and made explicit in Section~\ref{subsec_explicit_abelian}: 
it provides an isomorphism between the Hopf algebras~$\OO_k(G)$ and~$k[A]$. 
A bilinear form~$b$ on~$G$, viewed as a universal $R$-matrix for~$\OO_k(G)$, 
thus gives rise to a universal $R$-matrix for~$k[A]$; 
the latter is none other than~$R(A,b)$, as defined by~\eqref{Rb}.

When $b$ is alternating, we have proved in Section
\ref{subsec_explicit_abelian} that it corresponds to a two-cocycle
$c$, and in turn $c$ defines a twist $F$ on $k[A]$ such that $R_F =
R(A,b)$.  From Lemma~\ref{lem-u} we have the following.

\begin{lem}\label{lem-uR}
When $b$ is an alternating bilinear form on $\widehat{A}$, the
Drinfeld element of the universal $R$-matrix $R(A,b)$ is $1$. 
\end{lem}

\subsection{The map $\Theta$}\label{sub-support}

For any finite group~$G$,
let $\B(G)$ denote the set of pairs~$(A, b)$, where $A$ is an
abelian normal subgroup of~$G$ and $b : \widehat{A} \times \widehat{A}
\to k^{\times}$ is a $k^\times$-valued $G$-invariant non-degenerate
alternating bilinear form on~$\widehat{A}$.  The letter $\B$ is meant
to suggest either ``bilinear'' or ``braiding''. We see $\B(G)$ as a
pointed (in particular, non-empty) set, the distinguished element
being $(\{ 1 \}, 1)$, the trivial bilinear form on the trivial
subgroup. We write this element simply~$1$.

Let us define a map of sets 
\[
\Theta : \H^2_{\ell}(G) \to \B(G)
\]
as follows.  An invariant twist~$F$ on~$k[G]$ determines by
Formula~\eqref{twistedR} a universal $R$-matrix~$R_F=F_{21} F^{-1}$.
As observed in Section~\ref{subsec_braidings}, there is a unique
abelian subgroup~$A$ of~$G$ such that
\[
\kg\langle R_F\rangle = k[A] \, .
\] 
Moreover, $R_F$ corresponds to a bilinear form~$b$ on~$\widehat{A}$,
and $R_F = R(A,b)$.
We set $\Theta(F) = (A, b)$ and we call~$A$ the \emph{socle} of~$F$.
Note that if $A$ is the socle of~$F$, then $R_F$ belongs to~$k[A]
\otimes k[A]$; if moreover $F\in k[A] \otimes k[A]$, then we say that
$F$ is {\em supported} by $A$. (More on this in the next section).

\begin{lem}\label{lem-theta-well-defined} 
(a) Let $R$ be a universal $R$-matrix for~$k[G]$ such that 
$(g\otimes g) \, R  = R \, (g\otimes g)$ for all $g\in G$. 
If $A$ is the abelian subgroup of~$G$ such that $\kg\langle R\rangle = k[A]$, 
then $A$ is normal.

(b) The map~$\Theta$ is well-defined.
\end{lem} 

\begin{proof}
(a) The conjugation by $g\in G$ induces an automorphism of the Hopf
algebra~$\kg$ that sends the $G$-invariant tensor~$R$ to itself.
Therefore it sends~$\kg\langle R \rangle = k[A]$ to itself, which
implies that $A$ is a normal subgroup.

(b) We have observed above that $R_F$, and thus~$(A, b)$, depends
only on the class of~$F$ in~$\H^2_{\ell}(G)$.  
It follows from Part~(a) that $A$ is normal. 
By Lemma~\ref{lem-u}, the Drinfeld element~$u_{R_F}$ is equal to~$1$, 
and therefore the bilinear form~$b$ is alternating. It is minimal, and so also non-degenerate
by Proposition~Ê\ref{prop_minimal_is_degenerate}. It is obviously $G$-invariant. 
\end{proof} 

We now state our main theorem.

\begin{thm}\label{thm_theta} 
Assume that the ground field~$k$ is algebraically closed and of
characteristic zero. For any finite group~$G$
the map $\Theta : \H^2_{\ell}(G) \to \B(G)$ enjoys the following properties:

(a) The fibres of $\Theta : \H^2_{\ell}(G) \to \B(G)$ are the left
cosets of~$\Int_k(G)/\Inn(G)$.

(b) The subset $\Theta^{-1}(1)$ is a subgroup of $\H^2_\ell(G)$ 
isomorphic to $\Int_k(G) / \Inn(G)$. 

(c) If the order of~$G$ is odd, then any element $(A,b)$ in $\B(G)$
is of the form $\Theta(F)$ for an $F$ which is supported by $A$. 
In particular, $\Theta$ is surjective in this case.
\end{thm} 

\begin{proof}
(a) Assume that $\Theta(F) = \Theta(F')= (A,b)$. 
It follows that $R_F = R_{F'} = R(A,b)$, 
and by Proposition~\ref{RFdeterminesF}, the twists
$F^{-1}$ and $F'^{-1}$ are gauge equivalent:
\[
F^{-1} = (a\otimes a) \, F'{}^{-1} \, \Delta(a^{-1}) = (a\otimes a) \,
\Delta(a^{-1}) \, F'{}^{-1} \, .
\]
Since $F^{-1}$ and $F'^{-1}$ are invariant, we deduce from 
Proposition~\ref{prop-Hgaugequiv}\,(a) that $a$ belongs to~$N_k(G) = N(k[G])$. 
From~\eqref{abF} we obtain $(a\cdot (1\otimes 1))^{-1} = a^{-1}\cdot
(1\otimes 1)$, which can be rewritten as
$[(a\otimes a) \, \Delta(a^{-1})]^{-1} = (a^{-1} \otimes a^{-1}) \, \Delta(a)$. 
As a result,
\[
F = F' \, (a^{-1} \otimes a^{-1}) \, \Delta(a) \, .
\]
Thus $F$ and $F'$ belong to the same left coset of $\Int_k(G)/\Inn(G)$
(see Proposition~\ref{prop-Hgaugequiv}\,(c) together with
the isomorphisms~\eqref{eq-NH-Int} and~\eqref{eq-NG-Int}). 
The converse is clear.

(b) This follows from Part\,(a). 
Note that $\Theta^{-1}(1)$ is the subgroup isomorphic to~$\Int_k(G)/\Inn(G)$
mentioned in Corollary~\ref{coro_int_inn}\,(c).

(c) Assume that $G$ has odd order, and let $(A, b) \in \B(G)$. 
As observed in Section~\ref{subsec_explicit_abelian}, the fact that
$\widehat{A}$ has odd order allows us to pick the map 
$c : (x,y)\mapsto b(x/2, y)$ as a two-cocycle whose associated bilinear
form is $b$. Since $b$ is $G$-invariant, so is $c$. Therefore the
twist $F\in k[A]\otimes k[A]$ determined by~$c$ according to the
recipe of Section~\ref{subsec_explicit_abelian} is also invariant. 
We have $R_F = R(A,b)$, which implies that $\Theta(F) = (A,b)$.
\end{proof} 

Since the sets $\Int_k(G)/\Inn(G)$ and $\B(G)$ are finite, we deduce 
the following corollary
(we shall give an alternating proof in Section~\ref{sec_rationality}).

\begin{coro}\label{H2fini}
The group $\H^2_\ell(G)$ is finite.
\end{coro}

Theorem~\ref{thm_theta} has the following obvious consequences.

\begin{coro}\label{coro1}
If $\B(G) = \{ 1 \}$, then $\H^2_{\ell}(G) \cong \Int_k(G)/\Inn(G)$.
\end{coro}

\begin{coro}\label{coro2}
When $\Int_k(G) = \Inn(G)$, then $\Theta : \H^2_{\ell}(G) \to \B(G)$ is injective. 
If in addition $G$ has odd order, then $\Theta$ is a bijection of sets.
\end{coro}

\subsection{The set $\B(G)$ as a colimit}

If $A\subset B$ is an inclusion of abelian subgroups of~$G$, 
then the natural map
\[
H^2(\widehat{A}, k^\times)^G \to H^2(\widehat{B}, k^\times)^G
\]
is injective. 
Indeed, it suffices to check that 
$H^2(\widehat{A}, k^\times) \to H^2(\widehat{B}, k^\times)$ is injective.
Since $H^2(\widehat{A}, k^\times) \cong \Hom(\Lambda^2 \widehat{A}, k^\times)$
for all abelian groups~$A$,
this is equivalent to the surjectivity of
$\Lambda^2 \widehat{B} \to \Lambda^2 \widehat{A}$,
which follows from the surjectivity of the natural homomorphism 
$\widehat{B} \to \widehat{A}$ between the Pontryagin duals.

Now let $\mathcal{C}$ be the category whose objets are the
abelian normal subgroups of~$G$ and whose arrows are the inclusions. 
It follows that the colimit of $A \mapsto H^2(\widehat{A}, k^\times)^G$ over~$\mathcal{C}$ 
is essentially a union, which we denote by~$\bigcup H^2(\widehat{A}, k^\times)^G$.

As a matter of notation, given $b\in H^2(\widehat{A}, k^\times)^G$,
where $A$ is an abelian normal subgroup of~$G$, we shall write~$[A,b]$ 
for the corresponding element of $\bigcup H^2(\widehat{A}, k^\times)^G$. 
Of course several choices for $A$ and $b$ can lead to
the same element in the colimit.

\begin{lem}\label{lem-beta-r-matrix}
An element $\beta \in \bigcup H^2(\widehat{A}, k^\times)^G$ defines
a universal $R$-matrix $R(\beta )\in k[G]\otimes k[G]$ with Drinfeld
element $1$ by the formula $R(\beta) = R(A,b)$, where $A$ and $b$
are such that $\beta = [A,b]$.
\end{lem}

\begin{proof}
We need first to show that $R(A,b) = R(B,b')$ whenever $[A,b] = [B,b']$ in
the colimit. By the very construction of this colimit, it is enough to
check this when $A$, $B$, $b$ and $b'$ are as in Lemma
\ref{lem-restriction}. The same lemma yields the result,
and $R(\beta)$ is well-defined.
We use Lemma~\ref{lem-uR} to conclude that the Drinfeld element 
of~$R(\beta )$ is~$1$.
\end{proof}

\begin{prop}\label{lem-bg-is-colimit}
The map $\B(G) \to  \bigcup H^2(\widehat{A}, k^\times)^G$ that
sends~$(A,b)$ to~$[A,b]$ is a bijection of sets.
\end{prop}

\begin{proof}
We construct the inverse map. Given $\beta\in\bigcup
H^2(\widehat{A}, k^\times)^G$, let $R(\beta)$ be as in the previous
lemma, and let $A$ be defined by $\kg\langle R(\beta ) \rangle = k[A]$. 
It is clear that $R(\beta )$ is $G$-invariant; so we see from
Lemma~\ref{lem-theta-well-defined}\,(a) that $A$ is abelian and normal. 
The $R$-matrix $R(\beta)$ yields a bilinear form $b$ on
$\widehat{A}$. The form~$b$ is alternating since the Drinfeld element of~$R(\beta)$ 
is~$1$, it is non-degenerate by Proposition~\ref{prop_minimal_is_degenerate}, 
and it is $G$-invariant. Therefore, $(A,b) \in \B(G)$. 
We also observe from Lemma~\ref{lem-restriction} that $\beta = [A,b]$.
It is now clear that $\beta\mapsto (A,b)$ is the inverse of~$(A,b) \mapsto [A,b]$.
\end{proof}

In a sense the elements of $\B(G)$ provide canonical representatives
for the equivalence classes in the colimit. In the sequel we will
sometimes use the notation~$(A,b)$ in order to speak of the element~$[A,b]$ in the colimit; 
this notation carries the extra information that $b$ is non-degenerate
in accordance with Section~\ref{sub-support}.

The reader should beware that the notation $\bigcup H^2(\widehat{A}, k^\times)^G$ 
does not suggest any group law, for the colimit is taken in the category of sets.  
Nevertheless, we can at least construct a partially defined multiplication on it. 

Suppose given $(A,b)$ and $(B,b')$, and assume that $A$ and $B$ are
both contained in a larger abelian normal subgroup~$C$ of~$G$. 
Then $b$ and $b'$ can be viewed as elements of
$H^2(\widehat{C},k^\times)$ and composed in this group; their product
is still $G$-invariant and thus defines the element~$[C, bb']$. 
We put $[A,b]\cdot [B,b'] = [C, bb']$ whenever $C$ exists with the
aforementioned properties. Note that this definition does not depend
on the choice of~$C$: 
indeed if $C'$ is another choice, then $C\cap C'$ is yet another abelian normal subgroup 
of~$G$ containing both~$A$ and~$B$, and it has inclusion maps to~$C$ and~$C'$.

In terms of $\B(G)$ the above product can be, somewhat less conveniently,
described as~$(D, bb')$, where $D$ is determined by $\kg\langle R(C, bb') \rangle = k[D]$; 
here $D$ is a subgroup of~$C$ and we have written~$bb'$ for the unique element 
in $H^2(\widehat{D}, k^\times)$
that maps to the product of $b$ and $b'$ in $H^2(\widehat{C}, k^\times)$.

The following lemma gives a sufficient condition for the map~$\Theta$ 
to be compatible with this structure.

\begin{lem}\label{theta-product}
Let $F$ and $F'$ be invariant twists for $G$ and suppose that $F$
commutes with $R_{F'}$.  Under this assumption,

(a) we have $R_{FF'} = R_F \, R_{F'}$;

(b) if $\Theta(F)$ and $\Theta(F')$ can be composed in~$\B(G)$, 
then  $\Theta(F) \, \Theta(F') = \Theta(FF')$.
\end{lem}

\begin{proof}
(a) We simply compute: 
\begin{eqnarray*}
R_{FF'} & = & (FF')_{21} (FF')^{-1}
=  F_{21} F'_{21} (F')^{-1} F^{-1} \\
& = & F_{21} R_{F'} F^{-1}
= R_F R_{F'} \, . 
\end{eqnarray*}

(b) Let $\Theta(F) = (A,b) = \alpha $ and $\Theta(F') = (B, b') = \beta $. 
By assumption there exists an abelian normal subgroup~$C$ of~$G$ 
containing $A$ and~$B$. Let $\gamma = \alpha \beta = [C, bb']$. 
The $R$-matrices $R(\alpha)$, $R(\beta)$, and $R(\gamma)$ can all be
viewed as tensors in $k[C]\otimes k[C]$. Correspondingly, we view~$b$,
$b'$, and~$bb'$ as bilinear forms on~$\widehat{C}$. 
It is clear from~\eqref{Rb} that $R(C, b) \, R(C, b') = R(C, bb')$. 
However, $R(C, b) = R(\alpha)$ by definition, and similarly, 
$R(C, b')= R(\beta)$ and $R(C, bb') = R(\gamma )$. 

Note that for any invariant twist $J$ we always have $R( \Theta(J) )= R_J$ 
as a trivial consequence of the definitions.  This remark shows,
on the one hand, that we have $R_F = R(\alpha)$, $R_{F'} = R(\beta)$, and
$R_{FF'} = R( \Theta(FF') )$; on the other hand, combining\,(a) with
the equality $R(\alpha)\, R(\beta) = R(\gamma)$, which we have just
established, we deduce $R_{FF'} = R(\gamma)$.
The identity $R( \Theta(FF') ) = R(\gamma)$ clearly implies $\Theta(FF') = \gamma $.
\end{proof}

We can gain some information about the torsion in~$\H^2_\ell(G)$
from Lemma~\ref{theta-product}.

\begin{coro}\label{coro-torsion}
If $F$ commutes with~$R_F$, then the powers of~$\Theta(F)$ are well-defined.
Moreover, if $\Theta(F)$ has order~$n$, then the order of~$F$ in~$\H^2_\ell(G)$ 
is $nd$, where $d$ divides $\textnormal{card}( \Int_k(G) / \Inn(G) )$.
\end{coro}

\begin{proof}
The lemma and an immediate induction imply that 
$ \Theta(F^k) = \Theta(F)^k$
for any non-negative integer~$k$. Now apply Theorem~\ref{thm_theta}. 
\end{proof}

\begin{Rem}\label{rmk-odd-implies-torsion}
The condition that $F$ commute with~$R_F$ is trivially satisfied if
$F$ is supported by its socle (recall that this means $F\in k[A]\otimes k[A]$)
for $F$ and $R_F$ in this case both belong to
the commutative algebra $k[A]\otimes [A]$. 
Theorem~\ref{thm_theta} guarantees that any twist~$F$ is supported by its socle in the
particular case when $G$ has odd order and $\Int_k(G) = \Inn(G)$. 
The next corollary goes further with a very special class of examples.
\end{Rem}

\begin{coro} \label{coro-fixed-points}
Assume that $G$ is a group of odd order such that $\Int_k(G) = \Inn(G)$.
If $G$ has a unique maximal abelian normal subgroup~$A$, then
\[
\H^2_{\ell}(G) \cong H^2(\widehat{A}, k^\times)^G \, .
\]
\end{coro}

\begin{proof}
It follows from Theorem \ref{thm_theta} and Proposition
\ref{lem-bg-is-colimit} that $\H^2_\ell(G)$ is via the map~$\Theta $
in bijection with $H^2(\widehat{A}, k^\times)^G$ as a set.
Let us check that $\Theta $ is a homomorphism in this case.

Now, the same theorem shows that any element of $\H^2_\ell(G)$
can be represented by a twist~$F$ that is supported by its socle~$A_F$. 
Suppose that $F$ and~$F'$ are chosen in this way. 
By hypothesis, $A_F \subset A$ and $A_{F'} \subset A$, and so $F$ and $F'$ 
are both in~$k[A]\otimes k[A]$. Clearly then $F$ commutes with $R_{F'}$ since
$A$ is abelian. Lemma~\ref{theta-product} then implies that 
$\Theta(F) \, \Theta(F') = \Theta(FF')$, as desired.
\end{proof}

%% file: lazy_is_abelian.tex
\section{Some consequences of semi-simplicity}\label{sufficient}

The aim of this short section is to obtain a sufficient
condition on a finite group~$G$ for the group $\H^2_{\ell}(G)$ to be abelian
and to prepare the ground for the next section.
We keep assuming that~$k$ is algebraically closed.

In this situation, $k[G]$ is semi-simple and
there is an algebra isomorphism
\begin{equation}\label{decomp}
k[G] \cong \prod_{\rho \in \widehat G}\, \End_k(V_{\rho})\, , 
\end{equation}
where $\rho$ runs over the set~$\widehat G$ of irreducible representations of~$G$.
Consequently,
\begin{equation*}
k[G]^{\times} \cong \prod_{\rho \in \widehat G}\, \Aut_k(V_{\rho})\, , 
\end{equation*}

Since the isomorphism~\eqref{decomp} sends any element $a \in k[G]$ to the 
left multiplication by~$a$ on each representation space~$V_{\rho}$, it sends
the conjugation action of~$G$ on~$k[G]$ to the action of~$G$
on~$\End_k(V_{\rho})$ given by $(gf) : v\mapsto g f(g^{-1}v)$
for all $g \in G$, $f \in \End_k(V_{\rho})$.
It follows from this and from Schur's lemma
that the subalgebra $k[G]^G$ of $G$-invariants elements
is given by
\begin{equation}\label{decomp2}
k[G]^G \cong \prod_{\rho \in \widehat G}\, \End_G(V_{\rho}) \cong \OO_k({\widehat G}) \, .
\end{equation}

From the above we also deduce the $G$-invariant algebra isomorphisms
\[
k[G] \otimes k[G] \cong
\prod_{\rho, \sigma \in \widehat G}\, 
\bigl(\End_k(V_{\rho}) \otimes \End_k(V_{\sigma})\bigr)
\cong
\prod_{\rho, \sigma \in \widehat G}\, \End_k(V_{\rho} \otimes V_{\sigma}) \, .
\]
Decomposing the tensor products $V_{\rho} \otimes V_{\sigma}$
into simple modules
\[V_{\rho} \otimes V_{\sigma} \cong 
\bigoplus_{\tau \in \widehat G}\, V_{\tau}^{d_{\rho, \sigma}^{\tau}} \, ,
\]
we obtain the $G$-invariant algebra isomorphism
\begin{equation}\label{decomp21}
k[G] \otimes k[G] \cong
\prod_{\rho, \sigma \in \widehat G}\,
\End_k\Bigl(
\bigoplus_{\tau \in \widehat G}\, V_{\tau}^{d_{\rho, \sigma}^{\tau}} \Bigr) \, .
\end{equation}

We claim that the subalgebra $(k[G] \otimes k[G])^G$ of $G$-invariant
elements of $k[G] \otimes k[G]$ can be decomposed 
into a product of matrix algebras as follows:
\begin{equation}\label{decomp3}
(k[G] \otimes k[G])^G \cong 
\prod_{\rho, \sigma \in \widehat G}\,
\prod_{\tau \in \widehat G}\, M_{d_{\rho, \sigma}^{\tau}}(k)\, . 
\end{equation}
Indeed, taking $G$-invariants of both sides of~\eqref{decomp21}
and using Schur's lemma twice, we obtain
\begin{eqnarray*}
(k[G] \otimes k[G])^G 
& \cong &\prod_{\rho, \sigma \in \widehat G}\,
\End_G \Bigl( \bigoplus_{\tau \in \widehat G}\, 
V_{\tau}^{d_{\rho, \sigma}^{\tau}} \Bigr) \\
& \cong &\prod_{\rho, \sigma \in \widehat G}\,
\prod_{\tau \in \widehat G}\, 
\End_G(V_{\tau}^{d_{\rho, \sigma}^{\tau}}) \\
& \cong & \prod_{\rho, \sigma \in \widehat G}\,
\prod_{\tau \in \widehat G}\, 
M_{d_{\rho, \sigma}^{\tau}}(\End_G(V_{\tau})) \\
& \cong & \prod_{\rho, \sigma \in \widehat G}\,
\prod_{\tau \in \widehat G}\, 
M_{d_{\rho, \sigma}^{\tau}}(k)\, . 
\end{eqnarray*}
We thus see that the algebra $(k[G] \otimes k[G])^G$ is semisimple;
moreover, it is \emph{commutative} if and only if all multiplicities~$d_{\rho, \sigma}^{\tau}$
are equal to~$1$ or $0$. In this case we shall say that $G$ has \emph{no
multiplicities}. 

Observe that an abelian group has no multiplicities. 
Note also that a quotient of a group without multiplicities has no multiplicities. 

The following result provides a sufficient condition for $\H^2_{\ell}(G)$ to be abelian.

\begin{prop}\label{nomultiplicities}
If $G$ has no multiplicities, then $A^2(k[G])$ and $\H^2_{\ell}(G)$ are abelian groups.
\end{prop} 

\pf
Since $A^2(k[G]) = ((k[G] \otimes k[G])^G)^{\times}$, it follows from~\eqref{decomp3}
that
\begin{equation*}
A^2(k[G]) \cong \prod_{\rho, \sigma \in \widehat G}\,
\prod_{\tau \in \widehat G}\, 
GL_{d_{\rho, \sigma}^{\tau}}(k) \, .
\end{equation*}
Therefore, if $d_{\rho, \sigma}^{\tau} \leq 1$ for all $\rho, \sigma, \tau \in \widehat G$,
the group $A^2(k[G])$ is a $k$-split torus, which is abelian.
Since~$\H^2_{\ell}(G)$ is a subquotient of~$A^2(k[G])$, it is abelian as well.
\epf

\begin{Rem}\label{rmk-torsion-applies}
A consequence of Proposition~\ref{nomultiplicities} is that, when $G$ has no
multiplicities, any invariant twist $F$ commutes with~$R_F$: indeed
$F$ and $R_F$ both belong to~$A^2(k[G])$. 
In this case we may appeal to Corollary~\ref{coro-torsion} in order to study
the torsion in~$\H^2_\ell(G)$.
\end{Rem}

%% file: algebraic_groups.tex
\section{Rationality questions}\label{sec_rationality}

In this section 
we no longer assume that the ground field~$k$ is algebraically closed.  
Let~$\bar k$ be the algebraic closure of~$k$.
We fix a finite group $G$.

\subsection{The point of view of algebraic groups}

For any $k$-algebra $K$, and for $i= 1, 2, 3$, we define $A^i(K)$ to
be the group of $G$-invariant invertible elements in~$K[G]^{\otimes i}$, 
where $G$ acts diagonally by conjugation.  We can see $A^i$ as
a group scheme over $k$.  By Section~\ref{sufficient} this group
scheme is smooth since $A^i(K)$ is a product of general linear groups
$GL_n(K)$ whenever $K$ contains $\bar k$.  When $K = k$, we have
$A^i(k) = A^i(k[G])$ in the notation of Section~\ref{twist-coho}.

We define $\delta^1: A^1 \to A^2$ and $\delta^2_L, \delta^2_R: A^2 \to
A^3$ by the same formulas as in Section~\ref{twist-coho}; these are
morphisms of group schemes.  
The equalizer~$\Gamma$ of $\delta^2_L$ and $\delta^2_R$ is a smooth group scheme.
Moreover,
\[
	\Gamma(k) = Z^2(k[G]) 
\]
by Section~\ref{twist-coho}. Using Notation~2.1, we obtain
\begin{equation}\label{Gamma1}
\Gamma(k) / \delta^1( A^1(k) ) = \H^2_{\ell}(G/k) \, .
\end{equation}

\subsection{Tangent Lie algebras} 

Let $\Gamma^0$ be the identity component of $\Gamma$. We have the
following.

\begin{prop}\label{prop_lie_surj} 
The image of $\delta^1 : A^1(\bar k) \to A^2(\bar k)$ is
$\Gamma^0(\bar k)$. 
\end{prop} 

\begin{proof}
Since we work in characteristic $0$, we can use the machinery of Lie algebras.
For a group scheme $S$, we write $\Lie S$ for the Lie algebra tangent to $S$. 
For instance, $\Lie A^i(K)= (K[G]^{\otimes i})^G$ for $i=1,2,3$.

The tangent maps $\Lie \delta^1$, $\Lie \delta^2_L$, $\Lie \delta^2_R$ are
given by the following formulas: 
\begin{eqnarray*}
\Lie \delta^1(Y) & = & Y \otimes 1 + 1\otimes Y - \Delta(Y) \, , \\
\Lie \delta^2_L (X) & = & X \otimes 1 + (\Delta \otimes \id)(X) \, , \\
\Lie \delta^2_R (X) & = & 1 \otimes X  + (\id \otimes \Delta)(X) \, ,
\end{eqnarray*}
where $Y\in \bar k[G]^G$ and $X \in (\bar k[G] \otimes \bar k[G])^G$.
The Lie algebra $\Lie \Gamma(\bar k)$ is the subspace
of those tangent vectors $X\in \Lie A^2(\bar k)$ such that 
$\Lie \delta^2_L(X) = \Lie \delta^2_R(X)$.  
Moreover, the proposition will be proved once we show that 
such an $X$ must be of the form $\Lie \delta^1(Y)$ for some $Y$ tangent to~$A^1$. 
This basic translation from algebraic groups to Lie algebras, over an algebraically closed
field of characteristic $0$, is classical and follows for example from
the material in \cite[Chap.~II, \S\S~6--7]{borel}.

As it happens, the formulas above for the tangent maps make sense for
any $Y\in \bar k[G]$ and $X \in \bar k[G] \otimes \bar k[G]$, not
necessarily $G$-invariant.  So if we prove
that the following is an exact sequence
\begin{equation}\label{complex}
\begin{CD}
0 @>{}>> \bar k[G] @>{\Lie \delta^1}>> \bar k[G]^{\otimes 2} @>{\Lie \delta^2_R - \Lie \delta^2_L}>>  
\bar k[G]^{\otimes 3} \, , 
\end{CD} 
\end{equation}
then the proposition will follow by taking $G$-fixed points.
It is easy to check that the composite map $(\Lie \delta^2_R - \Lie \delta^2_L) \circ \Lie \delta^1 = 0$,
so that \eqref{complex} is a complex of vector spaces.
One also verifies that the map $\Lie \delta^1 : \bar k[G] \to \bar k[G]^{\otimes 2}$ is injective.
So in order to prove the exactness of~\eqref{complex}, 
it suffices to check it at~$\bar k[G]^{\otimes 2}$.
To achieve this, we count dimensions. 
The subspace~$V$ of $\bar k[G]^{\otimes 2}$ spanned by the elements $g \otimes h$ 
with $g \neq h$ has dimension $n^2 -n$, where $n$ is the order of~$G$. 
A simple computation shows that the map $\Lie \delta^2_R - \Lie \delta^2_L$ 
is injective on~$V$. Therefore, its rank is at least~$n^2 -n$.
We deduce that the kernel of~$\Lie \delta^2_R - \Lie \delta^2_L$ 
has dimension at most $n^2 - (n^2 -n) = n$,
which is the dimension of the image of~$\Lie \delta^1$. This shows that
the kernel of~$\Lie \delta^2_R - \Lie \delta^2_L$ is equal to the image of~$\Lie \delta^1$.
\end{proof} 

As a consequence, we obtain another proof of Corollary~\ref{H2fini}.

\begin{coro}
The group $\H^2_{\ell}(G/\bar k)$ is finite.
\end{coro} 

\begin{proof}
By Proposition~\ref{prop_lie_surj},
\begin{equation}\label{Gamma2}
\H^2_{\ell}(G/ \bar k) = \Gamma(\bar k) / \delta^1( A^1(\bar k) ) 
= \Gamma(\bar k) / \Gamma^0(\bar k)
= \pi_0(\Gamma) \, ,
\end{equation}
which is finite.
\end{proof} 

\subsection{The rationality exact sequence}

We are now able to compare $\H^2_{\ell}(G/k)$ and $\H^2_{\ell}(G/\bar k)$.

\begin{thm}\label{comparison}
Assume that all $\bar k$-representations of~$G$ can be realized over~$k$.
Then there is an exact sequence 
\[
1 \longrightarrow H^1(k, Z(G)) \longrightarrow \H^2_{\ell}(G/k) \longrightarrow 
\H^2_{\ell}(G/\bar k) \longrightarrow 1 \, ,
\]
where $H^1(k, Z(G))$ is the first Galois cohomology group of~$k$ with coefficients in the centre of~$G$. 
\end{thm} 

\begin{proof}
Taking $\Gal(\bar k/k)$-fixed points in the exact sequence 
\[
1 \longrightarrow \Gamma^0(\bar k) \longrightarrow \Gamma(\bar k) 
\longrightarrow \pi_0(\Gamma) \to 1
\] 
yields the exact sequence
\[
1 \longrightarrow \Gamma^0(k) \longrightarrow \Gamma(k) 
\longrightarrow \pi_0(\Gamma) \to H^1(k, \Gamma^0) \, .
\]
Dividing out the first two groups by the central subgroup~$\delta^1(
A^1(k) )$, and identifying terms using~\eqref{Gamma1}
and~\eqref{Gamma2}, we obtain the exact sequence
\begin{equation} \label{eq-rat-ex-seq1}
1 \longrightarrow \Gamma^0(k) / \delta^1( A^1(k) ) \longrightarrow
\H^2_\ell(G/k)  
\longrightarrow \H^2_\ell(G/\bar k) \longrightarrow H^1(k, \Gamma^0)  \to 1 \, .
\end{equation}

The hypothesis on $k$ implies in particular that $A^1$ is a $k$-split
torus.  By Proposition~\ref{prop_lie_surj}, $\Gamma ^0$ is also a
$k$-split torus, so that $H^1(k, \Gamma^0) = 1$ by Hilbert's
Theorem~90.  

Now take $\Gal(\bar k/k)$-fixed points again, this time in the exact
sequence
\[ 
1 \longrightarrow Z(G) \longrightarrow A^1(\bar k) \overset{\delta^1}{\longrightarrow} 
\Gamma^0(\bar k) \longrightarrow 1 \, .
\]
We obtain the exact sequence
\begin{equation} \label{eq-rat-ex-seq2}
1 \longrightarrow Z(G) \longrightarrow A^1(k) \overset{\delta^1}{\longrightarrow} 
\Gamma^0(k) \longrightarrow  H^1(k, Z(G)) \longrightarrow H^1(k, A^1) \, .
\end{equation}
By Hilbert's Theorem~90 again, $H^1(k, A^1) = 1$, so that
\[
\Gamma^0(k) / \delta^1( A^1(k)) \cong H^1(k, Z(G)) \, .
\]
This shows that the exact sequence of the theorem is~\eqref{eq-rat-ex-seq1}.
\end{proof} 

\begin{Rem}
From the sequences~\eqref{eq-rat-ex-seq1} and~\eqref{eq-rat-ex-seq2}
we deduce without any assumptions on~$k$
that the natural map $\H^2_\ell(G/k) \to \H^2_\ell(G/\bar k)$ is injective 
whenever the centre of~$G$ is trivial. 
If moreover $k$ happens to satisfy the hypothesis of Theorem~\ref{comparison}, 
then this map is surjective, too. 
\end{Rem}

\begin{Rem}
It follows from the proof of Theorem~\ref{comparison} that
any element of $H^1(k, Z(G))$ can be represented by a map
$\gamma_a : \Gal(\bar k/k) \to Z(G)$ of the form $\gamma_a(\sigma) = {}^{\sigma} a \, a^{-1}$
for some $a \in A^1(\bar k)$. 
The image of~$\gamma_a$ in~$\H^2_{\ell}(G/k)$ can be represented by the invariant twist
$\delta^1(a) = (a \otimes a) \, \Delta(a^{-1})$.
\end{Rem}

%% file: examples.tex

\section{Examples and computations}\label{examples}

In this section we again assume that the ground field is algebraically closed
and we compute $\H^2_\ell(G)$ for some finite groups,
as an application of Theorem~\ref{thm_theta} and its corollaries.

\subsection{Examples of groups with trivial $\H^2_{\ell}(G)$}\label{BG}

\begin{thm}\label{simple} Let $k$ be an algebraically closed field.
Then $H^2_{\ell}(G) = 1$ if $G$ belongs to the following list of
finite groups:

(i) the simple groups,

(ii) the symmetric groups~$S_n$,

(iii) the groups $\SL_n(\F_q)$,

(iv) the groups $\GL_n(\F_q)$ when $n$ is coprime to $q-1$.
\end{thm}

The simple groups include

(a) all \emph{alternating groups}~$A_n$, except $A_4$ (the latter
contains Klein's \emph{Vierergruppe} as a normal subgroup and will be treated
in Section~\ref{subsec-twists-of-A4}),

(b) all \emph{projective linear groups} $\PSL_n(\F_q)$, except
$\PSL_2(\F_2)$ and~$\PSL_2(\F_3)$; we have $\PSL_2(\F_2) \cong S_3$
and $\PSL_2(\F_3) \cong A_4$, which are not simple.

\pf 
By Corollary~\ref{coro1} it suffices to check that $\B(G)$ is trivial
and $\Int_k(G) = \Inn(G)$ for each of these groups.

Let us start with~$\B(G)$. Recall that it consists of
all pairs~$(A, b)$, where $A$ is an abelian normal subgroup of~$G$ and
$b$ is a $G$-invariant nondegenerate alternating bilinear form on~$\widehat{A}$.
By Proposition~\ref{nondeg_square},
any abelian group having a non-degenerate alternating bilinear form 
is necessarily of symmetric type. 
Any group~$G$ contains an abelian normal subgroup of symmetric type, namely the trivial group.  
If it does not contain any other abelian normal subgroup of symmetric type, then the set~$\B(G)$ is trivial.

So to prove that $\B(G)$ is trivial for the groups listed in the theorem,
it is enough to check that they do not contain any
non-trivial abelian normal subgroups of symmetric type.
This holds for

(i) the finite simple groups for obvious reasons;

(ii) the symmetric groups because the only proper normal subgroups
they contain are the alternating groups;

(iii) the groups $\GL_n(\F_q)$ and $\SL_n(\F_q)$ because their only abelian
normal subgroups are groups of roots of unity, which are
cyclic, hence not of symmetric type.

Let us now check that $\Int_k(G) = \Inn(G)$ for the above groups.
Since~$k$ is algebraically closed, it is enough by~\eqref{IntAut} to check that
$\Aut_{\rm c}(G) = \Inn(G)$.

Now, for any simple group~$G$, Feit and Seitz~\cite{FS} proved that
$\Aut_{\rm c}(G) = \Inn(G)$ using the classification of finite simple groups.

For the symmetric groups~$S_n$, the argument goes as follows.  Each
class-preserving automorphism necessarily maps any transposition to a
transposition.  Now it is well known (and a simple exercise) that any
automorphism of~$S_n$ that preserves transpositions is necessarily
inner. For $n \neq 6$ one can also use the fact that all automorphisms
of~$S_n$ are inner.

Let us now deal with~$\SL_n(\F_q)$. If $(n,q) = (2,2)$, then
$\SL_n(\F_q) = \PSL_n(\F_q) = S_3$ and we conclude as above. 
Now assume that $(n,q) \neq (2,2)$.  A class-preserving automorphism
$\varphi$ of~$\SL_n(\F_q)$ necessarily fixes each element of the
centre~$Z$ of the group, hence induces a class-preserving automorphism
of~$\PSL_n(\F_q) = \SL_n(\F_q)/Z$.  
If $(n,q) \neq (2,3)$, then the latter is simple and 
it follows from the above-mentioned result by Feit and Seitz that the
automorphism induced by~$\varphi$ on~$\PSL_n(\F_q)$ is inner.  
If $(n,q) = (2,3)$, then $\PSL_n(\F_q) \cong A_4$, whose automorphisms are all inner.
We are thus reduced to the case when $\varphi$ induces the identity
on~$\PSL_n(\F_q)$.  It follows that there is a group homomorphism $z :
\SL_n(\F_q) \to Z$ such that for all $g\in \SL_n(\F_q)$ we have
$\varphi(g) = z(g) g$.
Since $\varphi$ fixes the elements of~$Z$, we have $z(g) = 1$ for all
$g\in Z$.  Therefore, $z$ factors through~$\PSL_n(\F_q)$. The range
of~$z$ being abelian, it must factor through the abelianization
of~$\PSL_n(\F_q)$, which is trivial since $\PSL_n(\F_q)$ is simple and
not abelian.  Hence, $z(g) = 1$ for all $g\in \SL_n(\F_q)$. It follows
that $\varphi$ is the identity.

Since $n$ is coprime to the order $q-1$ of the cyclic
group~$\F_q^{\times}$, it follows that $\F_q^{\times}$ is uniquely
$n$-divisible and that any element $g\in \GL_n(\F_q)$ can be written
(uniquely) as $g = \lambda u$, where $\lambda$ is the unique
$n$-th root of the determinant of~$g$ and $u\in \SL_n(\F_q)$.  Let
$\varphi$ be a class-preserving automorphism of~$\GL_n(\F_q)$.  We
necessarily have $\det(\varphi(g)) = \det(g)$.  Therefore $\varphi$
preserves the subgroup~$\SL_n(\F_q)$. 
From the decomposition $g = \lambda u$, 
we easily deduce that the restriction of~$\varphi$
to~$\SL_n(\F_q)$ is class-preserving. It follows from the above
discussion that this restriction is an inner automorphism. We can
therefore assume that the restriction of~$\varphi$ to~$\SL_n(\F_q)$ is
the identity. Since it is also the identity on the central matrices,
it is the identity on~$\GL_n(\F_q)$.  We have thus proved that any
class-preserving automorphism of~$\GL_n(\F_q)$ is the conjugation by
some element of~$\SL_n(\F_q)$.  \epf

\subsection{The wreath product $\Z/p \wr \Z /p$ for $p$ odd}

It provides a first simple example of a group with non-trivial~$\H^2_{\ell}(G)$.

\begin{prop} The group $\H^2_{\ell}(\Z/p \wr \Z/p)$ is an elementary
abelian $p$-group of rank $(p-1)/2$.
\end{prop}

\pf For simplicity, we work over the field $\C$ of complex numbers.
If $G$ is the wreath product $\Z/p \wr \Z/p$, then we can write $G$ as a semi-direct product
$G = V \rtimes C$, where $C$ is a cyclic group of order~$p$, generated
by an element $\sigma$, and $V = \F_p[C]$, on which $C$ acts by left
multiplication.  

The abelian normal subgroups of $G$ are exactly the submodules of~$V$.
It will follow from Corollary \ref{coro-fixed-points} that
\[ \H^2_{\ell}(G) = H^2(\widehat{V}, k^\times)^G \, ,
\] if we can only prove that $\Aut_{\rm c}(G) = \Inn(G)$.

So let $\varphi$ be a class-preserving automorphism of~$G$; we wish to
prove that $\varphi$ is inner.  By composing $\varphi$ with a suitable
inner automorphism, we may assume that $\varphi(\sigma) = \sigma$, and
with this assumption we see that $\varphi$ induces an automorphism of
the $\F_p[C]$-{module}~$V$.  Now, $\sigma$ has only one eigenvector
$\varepsilon_1$ (up to a scalar), with eigenvalue~$1$. Therefore we
can find a basis $(\varepsilon_1, \ldots, \varepsilon_p)$ of~$V$ such
that the matrix of~$\sigma$ in this basis is given by a single Jordan
block.  Moreover, the commutant of a Jordan block is the algebra of
matrices of the form
\[ \left(\begin{array}{ccccc} a & b & c & \cdots & ~ \\ 0 & a & b & c
& \vdots \\ \vdots & \vdots & a & b & \ddots \\ 0 & ~& ~& \ddots &
\ddots
\end{array}\right)
\] 
The conjugates of $\varepsilon_i$ are of the form $\varepsilon_i + b
\varepsilon_{i-1}$ for $b\in\F_p$. As a result, if a linear map of~$V$
given by a matrix as above is to send each $\varepsilon_i$ to one
of its conjugates, then it must act like conjugation by $\sigma^b$
(that is, all the constants in the matrix are $0$ except for $a$ and
$b$, and $a=1$). This proves that $\varphi$ is inner.

Let us now compute $H^2(\widehat{V}, \C^\times)^G$. First we use the
exponential exact sequence
\[
\begin{CD} 0 \longrightarrow \Z \longrightarrow \C
\overset{\exp}{\longrightarrow} \C^\times \longrightarrow 0 \, ,
\end{CD}
\]
which yields an isomorphism $H^2(\widehat{V}, \C^\times) \cong
H^3(\widehat{V},\Z)$ (and likewise for any finite group).  Since $V$
is elementary abelian, it is well known that $H^*(\widehat{V},\Z)$ is
the subspace of~$H^*(\widehat{V},\F_p)$ of those elements
killed by the Bockstein $\beta : H^*(\widehat{V}, \F_p) \to
H^{*+1}(\widehat{V}, \F_p)$.  Recall that in low degrees we have the
following identifications:
\begin{equation}\label{eq-low-degree1}
H^2(\widehat{V}, \Z) = H^1(\widehat{V}, \C^\times) =
\Hom(\widehat{V}, \C^\times) = \widehat{ \widehat{V} } = V \, .
\end{equation}
(The first identification comes from the exponential exact sequence
again.)  Also we have
\begin{equation}\label{eq-low-degree2}
H^1(\widehat{V}, \F_p) = \Hom(\widehat{V}, \F_p) =
\Hom(\widehat{V}, \C^\times) = \widehat{ \widehat{V} } = V \, .
\end{equation}
Let us pick a basis $(e_1, \ldots , e_p)$ of~$V$ such that $\sigma
(e_i) = e_{i+1}$ (taking indices modulo~$p$).  Each element~$e_i$
gives rise, via~\eqref{eq-low-degree1}, to an element of
$H^2(\widehat{V}, \F_p)$ which we still denote by~$e_i$. 
Via~\eqref{eq-low-degree2} it can also be identified with an element of
$H^1(\widehat{V}, \F_p)$, which we denote by~$u_i$ in the sequel.  The
full mod~$p$ cohomology ring is then
\[ H^*(\widehat{V}, \F_p) = \F_p[e_1, \ldots, e_{p}]\otimes \Lambda
(u_1, \ldots, u_{p}) \, .
\]
Therefore $H^3(\widehat{V}, \F_p)$ is isomorphic to $V\otimes V \oplus
\Lambda^3 V$ as an $\F_p[C]$-module, or equivalently as an
$\F_p[G]$-module. We wish to determine the kernel of the Bockstein
restricted to the subspace of $G$-invariant elements in this module :
this will be $H^3(\widehat{V}, \Z)^G = H^2(\widehat{V},
\C^\times)^G$. Since the Bockstein is visibly injective on the
$\Lambda^3 V$ summand (see details on the action below), we focus on
$V \otimes V$. Let
\[
v_k = \sum_{i=1}^p \, u_i \, e_{i+k} \, , 
\] 
for $0 \le k \le p-1$ (again the indices are to be understood
modulo~$p$).  It is easy to see that the elements~$v_k$ span the
subspace of fixed points.

The Bockstein is a derivation that vanishes on each~$e_i$ and $\beta
(u_i) = e_i$. Because the elements $e_i$ commute with one another, we
see that $\beta(v_k) = \beta(v_{-k})$. On the other hand the elements
$\beta(v_k)$ for $0 \le k \le (p-1)/ {2}$ are linearly
independent. Finally, the kernel of $\beta$, restricted to the module
of invariant elements, has dimension~$(p-1)/{2}$.  \epf

\subsection{The group $\Z/2 \wr \Z/2$}\label{dihedral}

We turn our attention to the case $p=2$.  
This will complete the series of wreath products and will
illustrate our method for groups of even order.  
Note that $G = \Z/2 \wr \Z/2$ is the \emph{dihedral group} of order~$8$.

\begin{prop} The group $\H^2_{\ell}(\Z/2\wr \Z/2)$ is trivial.
\end{prop}

\pf 
We keep the notation of the previous section and we observe that 
$\Aut_{\rm c}(G) = \Inn(G)$ as before.  It then follows from~\eqref{IntAut} 
and Theorem~\ref{thm_theta} that $\H^2_{\ell}(G)$ injects into~$\B(G)$.

The first difference with the odd~$p$ case appears when we look for
normal abelian subgroups: there are exactly two maximal such
subgroups, namely $E_1 = V = \langle e_1, e_2 \rangle$ 
and $E_2 = \langle \sigma , e_1 e_2 \rangle$ 
(the analogue of~$E_2$ for odd~$p$ is not normal).  
Both $E_1$ and $E_2$ are elementary abelian $2$-groups of rank~$2$.

Let $A = \Z/2\times \Z/2$ and let $(\widehat{e}_1 , \widehat{e}_2 )$ 
be a basis of~$\widehat{A}$.  We shall write the group law
on~$\widehat{A}$ multiplicatively. One can see directly that there is
only one non-trivial non-degenerate alternating bilinear form~$b$ on
this group with values in~$\C^\times$, namely the one determined by
\[ 
b(\widehat{e}_1, \widehat{e}_1) = b(\widehat{e}_2, \widehat{e}_2)  = 1
\quad\text{and}\quad 
b(\widehat{e}_1, \widehat{e}_2) = b(\widehat{e}_2, \widehat{e}_1) = -1 \, .
\] 
(In other words, this form is the determinant when the subgroup 
$\{ 1, -1\}$ of~$\C^\times$ is identified with~$\F_2$.) The form~$b$ is
evidently invariant under all automorphisms of~$\Z/2 \times \Z/2$.  
It follows that $\B(G)$ has exactly three elements.  
Theorem~\ref{thm_theta} then implies that $\H^2_\ell(G)$ has at most three elements.

Now the reader can check that $G$ has a well-defined automorphism
fixing~$e_1e_2$ and exchanging $e_2$ and~$\sigma$; it also exchanges
$E_1$ and~$E_2$.  From this, it follows that if the image of
$\H^2_{\ell}(G)$ in $\B(G)$ had more than one element, then it would
have all three elements. Let us show that this is impossible.

One checks easily that $G$ has no multiplicities in the sense of Section~\ref{sufficient}. 
According to Remark~\ref{rmk-torsion-applies}, we can apply Corollary~\ref{coro-torsion},
from which it follows that the elements of $\H^2_\ell(G)$ have order
dividing~$2$. This is impossible in a group of order~$3$, 
so that finally $\H^2_\ell(G)$ is trivial. 
\epf

This is an example where the map $\Theta$ is not surjective. It seems
instructive to spell out the reason why the argument used in the proof
of Theorem~\ref{thm_theta}\,(c) does not work here.

\begin{lem}\label{lem-no-F}
Let $\Z/2 = \langle \sigma \rangle$ act on 
$\widehat{A}= \Z/2 \times \Z/2 = \langle \widehat{e}_1, \widehat{e}_2 \rangle$ 
by $\sigma(\widehat{e}_1) = \widehat{e}_{2}$.
Then there is no $\sigma$-invariant normalized two-cocycle $c$ on~$A$
whose associated form is~$b$.
\end{lem}

\begin{proof} 
Since $b(\widehat{e}_1,\widehat{e}_2)=-1$, we would have 
$c(\widehat{e}_1, \widehat{e}_2) = -c(\widehat{e}_2, \widehat{e}_1)$; 
by invariance of~$c$, we would also have 
$c(\widehat{e}_1, \widehat{e}_2) = c(\widehat{e}_2, \widehat{e}_1)$, 
which yields a contradiction.
\end{proof}

The reader will see that the lemma applies to $\widehat{E}_1$ and
$\widehat{E}_2$ under an appropriate choice of generators. 
It shows directly that the non-trivial elements in $\B(G)$ do not
correspond to invariant twists on the corresponding abelian
groups. The argument above is the only one we have in order to prove
that these elements of $\B(G)$ are not of the form~$\Theta(F)$.

\subsection{A group of order~$27$}

We turn to the example of 
$G=(\Z/3\times\Z/3)\rtimes\Z/3$,
where $C=\Z/3$ acts on~$V=(\Z/3)^2$ via a single Jordan block with eigenvalue~$1$. 
This group is somewhat intermediate between $\Z/3 \wr \Z/3$ and $\Z/2 \wr \Z/2$;
it provides an example of a group~$G$ 
whose abelian normal subgroups are \emph{not} contained in a maximal one
and at the same time has a non-trivial~$\H^2_\ell(G)$.

\begin{prop}
We have $\H^2_\ell(G) = \Z/3 \times \Z/3$. 
\end{prop}

\begin{proof}

Write $c$ for the generator of~$C$ and $(e_1, e_2)$ for a basis of~$V$
such that the action of~$c$ is given by a Jordan block. The element
$e_1$ is central, and from $ce_2 c^{-1} = e_1 e_2$ we deduce that any map
from $G$ to an abelian group has $e_1$ in its kernel. In particular,
$e_1$ belongs to all the normal subgroups of $G$. Thus we can read off
these normal subgroups of $G$ from the subgroups of $G/\langle e_1
\rangle$, which is an elementary abelian $3$-group of order~$9$.

In this way we find exactly four abelian normal subgroups of~$G$ of
square order, say~$E_i$, $1 \le i \le 4$: 
they are all elementary abelian $3$-groups of order~$9$, 
a basis being given in each case respectively by
$(e_1,e_2)$, $(e_1, c)$, $(e_1, e_2c)$, and~$(e_1, e_2c^2)$. 
In each case, we can view~$E_i$ as an $\F_3$-vector space of dimension~$2$, 
and $e_1$ is an eigenvector for the action of~$G$ with eigenvalue~$1$. 
One checks in all four cases that there are no other eigenvectors, 
or equivalently that the action of~$G$ factors through a copy of~$\Z/3$ 
acting via a single Jordan block with eigenvalue~$1$. 
In particular, the action preserves the determinant. 
It is easy to see that the same can be said of the
action of $G$ on the Pontryagin duals $\widehat{E}_i$.

As it turns out, the only alternating bilinear forms on~$\F_3^2$ are
proportional to the determinant (here we identify the group of third roots
of unity in~$\C^\times$ with~$\F_3$).  Therefore, on each~$\widehat{E}_i$
we have precisely two $G$-invariant non-degenerate alternating bilinear forms, 
namely the determinant and its opposite. In total, $\B(G)$ has
nine elements, taking the trivial one into account.

It is easy to check that $\Int_k(G) = \Inn(G)$. It then follows
from Theorem \ref{thm_theta} that $\H^2_\ell(G)$ has order~$9$. 
Moreover, Corollary~\ref{coro-torsion} also applies 
(see Remark~\ref{rmk-odd-implies-torsion})
and we deduce that the elements of~$\H^2_\ell(G)$ have order dividing~$3$. 
The result follows.
\end{proof}

\subsection{The alternating group $A_4$}\label{subsec-twists-of-A4}

This example is similar to $\Z/2 \wr \Z/2$, as the group $A_4$ is a
semi-direct product $V \rtimes C$, where as above $V$ is a copy of
$\Z/2 \times \Z/2$ generated by the permutations $e_1= (1,2)(3,4)$ and $e_2= (1,3)(2,4)$, 
while $C$ is a copy of $\Z/3$ generated by the permutation~$\sigma= (1,2,3)$. 
However, we now have the following result (where $b$ is as in Section~\ref{dihedral}).

\begin{lem}\label{inv_cocycle}
Let $\Z/3 = \langle \sigma \rangle$ act on $\widehat{V}=
\Z/2 \times \Z/2 = \langle \widehat{e}_1, \widehat{e}_2 \rangle$ by
$\sigma(\widehat{e}_1) = \widehat{e}_2$ and $\sigma(\widehat{e}_2) =
\widehat{e}_1\widehat{e}_2$. 
Then there does exist a $\sigma$-invariant normalized two-cocycle~$c$ 
on~$\widehat{A}$ whose associated form is~$b$.
\end{lem}

\begin{proof} We only need to set the value of
$\lambda = c(\widehat{e}_1, \widehat{e}_1)$, and 
$\mu =  c(\widehat{e}_1, \widehat{e}_2)$. 
All other values of~$c$ are dictated by the requirement that it be normalized, 
$\sigma$-invariant, and subject to the condition $c(y,x) = b(x,y)\, c(x,y)$. 
We proceed by inspection and set $\lambda = -1$, $\mu = 1$. 
We checked (with the help of a computer) that $c$ is indeed a two-cocycle.
\end{proof}

\begin{prop}\label{prop-a4} The group $\H^2_\ell(A_4)$ has order two.
\end{prop}

\begin{proof} It is an elementary exercise to prove that all
automorphisms of~$A_4$ are inner, so that $\Theta$ is injective. 
The subgroup~$V$ is the only abelian normal subgroup, and again we see
that $\B(G)$ has precisely two elements, the non-trivial one being
$(V,b)$ with~$b$ as above. Thus $\H^2_\ell(G)$ has at most two elements.

A non-trivial twist~$F$ is afforded by Lemma~\ref{inv_cocycle}. 
Indeed, applying the recipe of Section~\ref{subsec_explicit_abelian},
assuming that $\widehat{e}_i$ is taken to be the character of~$V$ 
whose kernel is generated by~$e_i$, 
and writing $e_3 = e_1e_2$, the cocycle~$c$ exhibited in the proof of the lemma yields
\begin{eqnarray*}
4F 
& = & 1 \otimes 1 - (e_1 \otimes e_1 +  e_2 \otimes e_2 + e_3 \otimes e_3) \\
&& +  (1 \otimes e_1 +  e_1 \otimes 1) + 
(1 \otimes e_2 +  e_2 \otimes 1) + (1 \otimes e_3 +  e_3 \otimes 1) \\
&& +  (e_1 \otimes e_2 - e_2 \otimes e_1) 
+ ( e_2 \otimes e_3 -  e_3 \otimes e_2)
+ (e_3 \otimes e_1-  e_1 \otimes e_3) \, .
\end{eqnarray*}
By construction $F$ is $A_4$-invariant and maps to~$(V,b)$ under the map~$\Theta$.
\end{proof}

\subsection{A group considered by G. E. Wall}\label{Wall}

Our final example will be the group~$G$ generated by three generators
$s$, $t$, $u$ subject to the relations
\[
s^2=t^2=u^8= 1 \, , \quad st=ts \, , \quad sus^{-1} = u^3 \, , \quad 
tut^{-1} = u^5 \, .
\] 
It is isomorphic to the semi-direct product $\Z/8 \rtimes \Aut(\Z/8)$.

This group was considered by Wall \cite{Wal}, who proved that 
$\Aut_c(G) / \Inn(G) = \ZZ/2$, 
the non-trivial element being represented by 
the automorphism~$\alpha $ 
given by $\alpha (s) = u^4s$, $\alpha(t) = u^4t$, and $\alpha(u) = u$.

By~\eqref{IntAut} the automorphism~$\alpha$ is the conjugation
by an element $a$ of the normalizer~$N_k(G)$ of~$G$ in~$k[G]^{\times}$.
A computer search for~$a$ yielded the element 
\[
a = \frac{1}{2} \, (1 + u^4) + \frac{\sqrt 2}{4} \, u \, (1 - u^2 - u^4 + u^6) \, .
\]
One checks directly that $a^2 = 1$ (so~$a$ is invertible) and
\[
as =u^4s a \, , \quad at=u^4ta \, , \quad au = ua \, , 
\]
so that conjugation by~$a$ does indeed induce the automorphism~$\alpha $.

It follows from Proposition~\ref{prop-Hgaugequiv} that $\H^2_\ell(G)$
has a subgroup of order~$2$, generated by $F = (a\otimes a) \, \Delta(a^{-1})$.
The symmetric invariant twist~$F$ is determined by
\begin{eqnarray*}
8 F \!\!\!\!\!\!\!\!\!
&&  \!\! 
= 2 \, (u_{00} + u_{44}) + (u_{11} + u_{33} + u_{55} + u_{77}) \\
&& {} + u_{01} + u_{03} + u_{04} + u_{05} + u_{07}
+ u_{12} + u_{17} + u_{25} + u_{35} + u_{36} + u_{67} \\
&& {} + u_{10} + u_{30} + u_{40} + u_{50} + u_{70}
+ u_{21} + u_{71} + u_{52} + u_{53} + u_{63} + u_{76} \\
&& {} - (u_{13} + u_{14} + u_{15} + u_{16} + u_{23} + u_{27}
+ u_{34} + u_{37} + u_{45} + u_{47} + u_{56} + u_{57}) \\
&& {} - (u_{31} + u_{41} + u_{51} + u_{61} + u_{32} + u_{72}
+ u_{43} + u_{73} + u_{54} + u_{74} + u_{65} + u_{75}) \, ,
\end{eqnarray*}
where we set $u_{ij} = u^i \otimes u^j$ ($i,j \in \{0, 1, \ldots, 7\}$).

Let us now describe~$\B(G)$. The group $G$ has only one abelian
normal subgroup~$A$ of square order, 
namely the copy of $\Z/2\times \Z/2$ generated by~$t$ and~$u^4$. 
On~$\widehat{A}$ there is only one non-degenerate alternating bilinear form~$b$, 
and it is invariant under all automorphisms of~$\widehat{A}$, 
as already observed a couple of times. 
As a result, $\B(G)$ has exactly two elements (counting the trivial one).

It follows from Theorem \ref{thm_theta} and from $\Int_k(G)/\Inn(G) = \ZZ/2$
that $\H^2_\ell(G)$ has order~$4$ or~$2$
according as $\Theta$ is surjective or not. 
In this instance we were not able to conclude.  We can at least observe that
the action of~$G$ on~$A$ is such that Lemma~\ref{lem-no-F} applies,
from which we deduce that there is no invariant twist $F\in k[A]\otimes k[A]$ 
such that $\Theta(F) = (A,b)$. However, we do not know whether
there is an $F\in k[G]\otimes k[G]$ with this property.